\documentclass[a4paper]{article}

\usepackage{amssymb,mathtools}
\usepackage{amsthm}
\usepackage{amsmath}
\counterwithin{equation}{section}

\usepackage{graphicx,float,enumitem}
\usepackage{bm}
\usepackage[all]{xy}
\graphicspath{{figures/}}

\usepackage[default,scale=1]{opensans}
\usepackage[T1]{fontenc}

\usepackage[utf8]{inputenc}
\DeclareUnicodeCharacter{00A0}{ }
\usepackage{geometry}
\usepackage[hidelinks]{hyperref}
\usepackage[nameinlink]{cleveref}

\usepackage{crossreftools}
\let\ORGhypersetup\hypersetup
\protected\def\hypersetup{\ORGhypersetup}
\newcommand{\diffc}{D}

\let\vec\relax
\usepackage{bm}
\newcommand{\vec}[1]{\bm{#1}}

\newcommand{\Lipschitz}{{Lip}}

\DeclareMathOperator{\N}{N}

\newcommand{\R}{{\mathbb R}}
\newcommand{\Rd}{{\R^d}}

\newcommand{\diver}{\nabla \cdot}

\DeclareMathOperator{\sign}{sign}
\DeclareMathOperator{\supp}{supp}
\DeclareMathOperator{\Lip}{Lip}

\newcommand*\diff{\mathop{}\!\mathrm{d}}

\usepackage{mathtools}
\usepackage{etoolbox}

\DeclareFontFamily{U}{matha}{\hyphenchar\font45}
\DeclareFontShape{U}{matha}{m}{n}{
	<-6> matha5 <6-7> matha6 <7-8> matha7
	<8-9> matha8 <9-10> matha9
	<10-12> matha10 <12-> matha12
}{}
\DeclareSymbolFont{matha}{U}{matha}{m}{n}

\DeclareFontFamily{U}{mathx}{\hyphenchar\font45}
\DeclareFontShape{U}{mathx}{m}{n}{
	<-6> mathx5 <6-7> mathx6 <7-8> mathx7
	<8-9> mathx8 <9-10> mathx9
	<10-12> mathx10 <12-> mathx12
}{}
\DeclareSymbolFont{mathx}{U}{mathx}{m}{n}

\DeclareMathDelimiter{\vvvert} {0}{matha}{"7E}{mathx}{"17}%

\DeclarePairedDelimiterX{\normiii}[1]
{\vvvert}
{\vvvert}
{\ifblank{#1}{\:\cdot\:}{#1}}

\newtheorem{theorem}{Theorem}[section]
\newtheorem{proposition}[theorem]{Proposition}%
\newtheorem{corollary}[theorem]{Corollary}%
\newtheorem{lemma}[theorem]{Lemma}%

\theoremstyle{definition}
\newtheorem{definition}[theorem]{Definition}%
\newtheorem{remark}[theorem]{Remark}%

\usepackage[
url=false,
isbn=false,
style=numeric,
sortcites=true,
maxbibnames=100,
maxcitenames=6, %
giveninits,
sorting=nyt,
doi=true,
uniquename=false,
]{biblatex}
\renewbibmacro{in:}{}
\DeclareFieldFormat[article]{citetitle}{#1}
\DeclareFieldFormat[article]{title}{#1}

\makeatletter
\renewcommand*{\@fnsymbol}[1]{\ensuremath{\ifcase#1\or \star \or \dagger\or \ddagger\or
		\mathsection\or \mathparagraph\or \|\or **\or \dagger\dagger
		\or \ddagger\ddagger \else\@ctrerr\fi}}
\makeatother

\title{
	Interpreting systems of continuity equations in spaces 
	\\ 
	of probability measures through PDE duality}
\author{
	Jos\'e A. Carrillo%
	\thanks{Mathematical Institute, University of Oxford, Oxford OX2 6GG, UK.  \href{mailto:carrillo@maths.ox.ac.uk}{carrillo@maths.ox.ac.uk}} %
	\and 
	David Gómez-Castro%
	\thanks{Mathematical Institute, University of Oxford, Oxford OX2 6GG, UK}
    \thanks{Departamento de Matemáticas, Universidad Autónoma de Madrid, 28049 Madrid, Spain.  \href{mailto:david.gomezcastro@uam.es}{david.gomezcastro@uam.es}}
}

\begin{document}

\maketitle

\begin{abstract}
We introduce a notion of duality solution for a single or a system of transport equations in spaces of probability measures reminiscent of the viscosity solution notion for nonlinear parabolic equations. Our notion of solution by duality is, under suitable assumptions, equivalent to gradient flow solutions in case the single/system of equations has this structure. In contrast, we can deal with a quite general system of nonlinear non-local, diffusive or not, system of PDEs without any variational structure.


\end{abstract}

\section{Introduction}
Many PDE modelling instances of applied analysis lead to transport equations for a \textit{density function} $\rho$ of the form $\partial_t \rho = \diver (\rho \vec v)$, where $t> 0$, $x \in \Rd$, and $\vec v$ is the \textit{velocity field}. 
In many of these applications, there is a ubiquitous choice of convolutional-type velocities depending on the density well motivated by applications in mathematical biology, social sciences and neural networks, see for instance \cite{ME99,TBL06,APS06,MT15,CHS18,CMSTT19,FernandezRealFigalli}. 
Moreover, in many of these applications, $\bm v$ is given by a function of $\rho$ itself, and it may happen that the density $\rho$ is not an integrable function, but rather a measure. Since the mathematical treatment of these problems is rather tricky, and indeed many ``non-physical'' solutions may appear, the natural way to approach these problems is the so-called vanishing-viscosity method. With this method linear diffusion $D \Delta \rho$ is included to the right-hand side, leading to the problem
\begin{equation*}
    \partial_t \rho = \diver (\rho \vec v[\rho]) + D \Delta \rho.
\end{equation*}
The ``physical'' or \textit{entropic} solution is the one obtained in the limit $D \searrow 0$. In this work, we deal with systems of such equations, where several \emph{species} inhabit a domain. We will consider that the evolution of these species is coupled only by the velocity field. We are interested in systems of $n$ aggregation-diffusion equations of the form
\begin{gather}
\label{eq:main}
	\partial_t \mu^i_t = \diver \left( \mu^i_t {\bm K}^{i}_t [\vec \mu_t]   \right) + \diffc^i \Delta \mu_t^i , \qquad i = 1, \cdots, n \,\text{ and }\, (t,x) \in (0,\infty) \times  \Rd
\end{gather}
where $\diffc^i \ge 0$ and $\bm K^i_t$ are velocity-fields depending non-locally on the full vector of densities $\vec \mu = (\mu^1, \cdots, \mu^n)$ and possibly in the spatial or time variables, under some hypotheses described below.
In fact, we are to work with a generalisation allowing $K_t$ to depend on the previous time states see \eqref{eq:main generalised}.
Notice that we cover the cases with and without diffusion ($D_i>0$ and $D_i =0$) in a unified framework.
We will introduce a new notion of solution, which we call dual-viscosity solution, with well-posedness, and that is able to detect this vanishing-viscosity limit.

A particularly interesting case is when the problem \eqref{eq:main} has a 2-Wasserstein gradient-flow structure. For example, in the scalar setting ($n = 1$) without diffusion ($D=0$) and with an interaction potential energy functional of the form
\begin{equation*}
	{ \bm K} [ \mu] = \nabla \frac{\delta \mathcal F}{\delta \rho} [ \mu], \qquad \mathcal F [\mu] = \frac 1 2  \int_\Rd \int_\Rd W(x-y) \diff \mu (x) \diff \mu (y) ,
\end{equation*}
the classical gradient flow theory developed in \cite{Ambrosio2005, AmbrosioSavare2007} applies when $W$ is $\lambda$-convex and smooth, and this later generalised to different settings \cite{CDFLS11,BLL12,DiFrancesco2013,BCLR13,JV16,CT16,CKY18,CDEFS18,DES21,EPSS21} for instance. This can also be applied to \eqref{eq:main} with $D > 0$ by including the Boltzmann entropy in the free energy functional, see  \cite{CarrilloMcCannVillani2003,CarrilloMcCannVillani2006,Tug13,ZM15,BDZ17,CGPS20,MF20,CGYZ22}. We refer to \cite{CCY19} for a recent survey of results in aggregation-diffusion equations. We show in \Cref{sec:gradient flows} that our new notion of dual viscosity solutions is equivalent, in some examples, to the notion of gradient flow solutions. 

In many situations, a theory of entropy solutions in the sense of Kruzkov (see \cite{Kruzkov1970,Carrillo1999}) is possible, when we have initial data and solutions in $L^1 (\Rd) \cap L^\infty (\Rd) \cap BV (\Rd)$. For example, the work \cite{DiFrancesco2019} deals with $d = n = 1$, $\diffc = 0$ and $K =\theta(\mu)  W' * \mu$ where $\theta(M) = 0$ for some $M > 0$, and $W$ very regular. Then, clearly initial data $0 \le  \rho_0 \le M $ lead to solutions $0 \le \rho_t \le M$. The authors prove uniqueness of entropy solutions in the sense of Kruzkov in the space $L^\infty (0,T; L^1 (\Rd) \cap L^\infty (\Rd) \cap BV(\Rd))$, by using the continuous dependence results in \cite{Karlsen2003} and a fixed point argument. This requires fairly strict hypothesis on the regularity of the initial data (namely bounded variation). They also introduce an interesting constructive method based on sums of characteristics of domains evolving in time. A recent extension of \cite{DiFrancesco2019} to related problems can be seen in \cite{FagioliTse2022}.

However, in the diffusion-less regime $D = 0$, these models based on transport equations can describe particle systems described by Dirac deltas, which interact through an ODE. The idea of working with Dirac deltas leads to numerical methods for general integrable data by approximating the initial datum by a sum of Diracs and regularizing the entropy \cite{CB16,CCP19,CHWW20}. 
Kruzkov's notion of entropy solutions cannot be extended to measures.
Furthermore, for less regular $W$ (or $\bm K [ \mu ] = \nabla V(x)$ known), the finite-time formation of Dirac deltas even for smooth initial data is not known. The paper \cite{DiFrancesco2013} studies uniqueness of distributional solutions for $n = 2$, when $\diffc^i = 0$, 
\begin{equation} 
	\label{eq:set-up DiFrancesco2013}
	\bm K^{i} [\vec \mu] = \nabla \mathcal H_i * \mu^i + \nabla \mathcal K_i * \mu^j , \qquad \{j \} = \{1,2 \} \setminus \{i\}
\end{equation}
where $\mathcal H_i, \mathcal K_i \in \mathcal W^{2,\infty}$.
The authors indicate that any distributional solution is a push-forward solution (pointing to \cite[Chapter 8]{Ambrosio2005}) of the flow that solves
\begin{equation}
	\frac{\diff X^i_t}{\diff t} = \bm K^i_t [ \vec  \mu_t   ] (X^i_t), \qquad X_0^i (y) = y.
\end{equation}
This push-forward structure is not available for problems with diffusion. The main goal of this work is to construct a notion of well-posed solutions that works for the case 
\begin{equation}
\label{eq:gradient flow of convolution}
    {\bm K}^{i} [\vec \mu] = \nabla  \sum_{j=1}^n W_{ij} * \mu^j,
\end{equation}
$d, n \ge 1$, $D \ge 0$, and measure initial data. We will work in $1$-Wasserstein space, and recover con\-ti\-nuous dependence estimates by discussing the ``dual'' problem, for which we use viscosity solutions by studying the corresponding Lipschitz estimates. We show also how our results can be adapted to different spaces (e.g., $\dot H^{-1}$ and $H^1$).

So far, the velocity field depends only on $t$ and $\vec \mu_t$. We can consider a  
generalisation of $K$, denoted $\mathfrak K$, that depends
also on the past.
We denote this extended notation by $\mathfrak K [\overline \mu]_t$.
By saying that $\mathfrak K$ depends only on the past we mean that $0 \le t \le s \le T$ and $\vec \mu \in C([0,T]; \mathcal P_1 (\Rd)^n)$ then
\begin{equation}
\tag{H}
\label{eq:continuation}
   ({\bm {\mathfrak K}}^{i}[\vec \mu|_{[0,s]}])_t = ({\bm {\mathfrak K}}^{i}[\vec \mu|_{[0,t]}])_t.
\end{equation}
To simplify the notation, when it will not lead to confusion,
we simply write $\mathfrak K[\mu]_t$.
For example, we cover the case
$$
    \mathfrak K[\mu]_t = \int_0^t \nabla W_s * \mu_s \diff s,
$$
where $W_t$ are a family of convolution kernels. 
Hence, we extend the \eqref{eq:main} to the more general
\begin{gather}
\tag{$\mathrm{P}$}
\label{eq:main generalised}
    \partial_t \mu^i_t = \diver \Big( \mu^i_t ({\bm {\mathfrak K}}^{i}[\vec \mu])_t  \Big) + \diffc^i \Delta \mu_t^i , \qquad i = 1, \cdots, n \,\text{ and }\, (t,x) \in (0,\infty) \times  \Rd.
\end{gather}
This covers the previous setting by defining $$ (\bm {\mathfrak K}[\vec \mu])^i _t = {\bm K}_t^i[\vec \mu_t].$$

The structure of the paper is as follows.
\Cref{sec:main results} is devoted to the main results and methods. We begin in \Cref{sec:definition of dual viscosity solution} by motivating and introducing a new notion of solution in $1$-Wasserstein space, motivated by the duality with viscosity solution of the dual problem, which we call \textit{dual viscosity solution}. In \Cref{sec:theorem statement K regular} we present well-posedness results for ${\bm K}$ regular enough. In \Cref{sec:thm stability} we present a result of stability under passage to the limit. In \Cref{sec:Newtonian} we discuss a limit case of the regularity
assumptions of the well-posedness theory,
given by the Newtonian potential. We close the presentation of the main results with \Cref{sec:open problem} on open problems our current results do not cover, mainly non-linear diffusion and non-linear mobility.
The full proofs are postponed to the next sections. First, in \Cref{sec:solutions of PE*} we prove well-posedness and estimates for the dual problem of \eqref{eq:main}, where we assume ${\bm K}[\mu]$ is replaced by a known field. Then, in \Cref{sec:main E known} we prove well-posedness of \eqref{eq:main} when ${\bm K}[\mu]$ 
is again replaced 
by a known field. In \Cref{sec:proof of main theorem} we prove the results stated in \Cref{sec:theorem statement K regular} by a fixed-point argument based on the previous estimates.

We conclude the paper with two sections on comments and extensions of our main results. First, in \Cref{sec:gradient flows}, we show that, where a gradient-flow structure is available, our solutions coincide with the steepest descent solutions, under some assumptions. 
In \Cref{sec:well-posedness in H-1} we extend our main results from the $1$-Wasserstein framework to the negative homogeneous Sobolev space $\dot H^{-1}$. We recall that $\dot H^{-1}$ is related to the $2$-Wasserstein space (see below).

\section{Main results and methods}
\label{sec:main results}
 
\subsection{A new notion of solution via duality}
\label{sec:definition of dual viscosity solution}

Considering first the case $n = 1$, let ${\bm E}_t = - {\bm {\mathfrak K}}[\mu]_t$, so that we are looking at the problem
\begin{gather}
\tag{P$_{\bm E}$}
\label{eq:main E known}
	\partial_t \mu_t = -\diver \left( {\bm E}_t \mu_t  \right) + \diffc \Delta \mu_t.
\end{gather}
Furthermore, if we naively assume the equation to be satisfied pointwise, we can multiply by a test function $\varphi$ and assume we can integrate by parts integrate by parts, we recover
\begin{equation}
\label{eq:weak form 1}
	\int_\Rd \varphi_T \diff \mu_T+ \int_0^T \int_\Rd \left( -\frac{\partial \varphi_t}{\partial t} - {\bm E}_t \nabla \varphi_t - \diffc \Delta \varphi_t    \right) \diff \mu_t \diff t = \int_\Rd \varphi_0  \diff \mu_0.
\end{equation} 
To cancel the double integral, we take $\varphi_t = \psi_{T-t}$ and $\psi_s$ a solution of
\begin{gather}
\label{eq:dual E known}
\tag{P$_{\bm E}^*$}
	    \frac{\partial \psi_s}{\partial s}  = {\bm E}_{T-s}  \nabla \psi_s + \diffc \Delta \psi_s  , \qquad  s \in [0,T], 
\end{gather} 
and $\psi_s$ Lipschitz in $x$.
Using $\varphi_t = \psi_{T-t}$ we reduce \eqref{eq:weak form 1} to the equivalent formulation
\begin{gather}
\tag{D$_{\bm E}$}
\label{eq:duality E known}
    \int_\Rd \psi_0 \diff \mu_T = \int_\Rd \psi_T \diff \mu_0.
\end{gather}
This is an interesting formulation because it will allow us to exploit the duality properties of the spaces 
for $\mu$ and $\psi$.

Hence, we are exploiting the well-known duality between the continuity problem \eqref{eq:main E known} and its dual \eqref{eq:dual E known}.

When $\diffc = 0$, this connects with the push-forward formulation for first order problems. Problem \eqref{eq:dual E known} can be solved by a generalised characteristics field $X_t$ (e.g., \cite[Section 3.2, p.97]{Evans1998}), and \eqref{eq:duality E known} means precisely that $\mu_T $ is the push-forward by the same field (see, e.g., \cite[Section 5.2, p.118]{Ambrosio2005}).

Notice that in the general setting we have $n$ equations, and $\psi$ depends on $T$. Hence, we denote our test functions $\psi^{T,i}$. We consider the system of $n$ dual problems
\begin{gather}
\label{eq:dual}
\tag{P$^*_i$}
\begin{dcases}
    \frac{\partial \psi_s^{T,i}}{\partial s}  = - ({\bm {\mathfrak K}}^i[\vec \mu|_{[0,T-s]} ])_{T-s}  \nabla \psi_s^{T,i} + \diffc^i \Delta \psi^{T,i}_s   & \text{for all } s \in [0,T], y \in \Rd  \\
    \psi_0^{T,i} = \psi_0^i.
\end{dcases}
\end{gather}
We also recover $n$ duality conditions
\begin{gather}
	\label{eq:duality}
	\tag{D$_i$}
	\int_\Rd \psi^{T,i}_0 \diff \mu^i_T  = \int_\Rd \psi^{T,i}_T \diff \mu^i_0, \qquad i = 1, \cdots, n.
\end{gather}

\begin{remark}
    Notice that we have not introduced a condition as $|x| \to \infty$ for $\psi$. 
    In the formal derivation of the notion of solution, we have required that we can formally 
    integrate
    by parts.
    This is also the standard argument to define distributional solutions.
    {%
        As with distributional solutions, we set \eqref{eq:duality} as our definition and prove that this is mathematically sound in terms well-posedness.
        We do not claim the PDE is satisfied in any stronger sense.
        We will make a comment on gradient flow solutions.
        Furthermore,%
    }
    we will not rigorously integrate by parts at any point in this manuscript.
    We will actually discuss in detail the precise admissible initial data $\psi_0$ for \eqref{eq:dual E known}.
    When we study solutions in $1$-Wasserstein space we will only assume that $\psi$ are Lipschitz (but not bounded). 
    The uniqueness of solutions of \eqref{eq:dual E known} in that setting is discussed in \Cref{sec:solutions of PE*}.
    In \Cref{sec:well-posedness in H-1} we extend the results to negative Sobolev spaces, and thus we will use some decay (see
    \Cref{rem:dot H -1 compact support}).
\end{remark}

\begin{remark}
	Notice that the problems for the different $\psi^i$ are de-coupled from each other. 
\end{remark}

It is well-known that classical solutions of \eqref{eq:dual E known} may not exist, specially in the case $\diffc = 0$. The approach developed by Crandall, Ishii and Lions to deal with limit is the notion of viscosity solution (see, e.g., \cite{crandall+ishii+lions1992users-guide-viscosity}), which we introduce below. 
Since the theory of viscosity solutions is usually well-posed for Lipschitz functions $\psi$, a natural metric for $\mu$ is the $1$-Wasserstein distance
(see, e.g., \cite[p. 207]{Villani2003}).
We recall that the space $\mathcal P_1 (\Rd)$ is the space of the probability measures $\mu$ such that
$$
    \int_\Rd |x| \diff \mu (x) < \infty.
$$
The distance $d_1$ that makes it a complete metric space is usually constructed by a Kantorovich optimal transport problem.
The reason we can exploit the duality between \eqref{eq:main E known} and \eqref{eq:dual E known} is the following well-known duality characterisation (see \cite[Theorem 1.14]{Villani2003}): if $\mu, \widehat \mu \in \mathcal P_1 (\Rd)$ 
then
\begin{equation}
\label{eq:duality characterisation of d1}
	d_1 (\mu, \widehat \mu) = \sup \left \{  \int_\Rd \psi \diff (\mu - \widehat \mu)  : \psi \in C (\Rd) \text{ such that }\Lip(\psi) \le 1    \right \},
\end{equation} 
where, here and below
\begin{equation*}
    \Lip(\psi) = \sup_{x \ne y } \frac{  |\psi(x) - \psi(y)|  }{|x-y|}.
\end{equation*}

We recall the notion of viscosity solution.
\begin{definition}
    We say that $\psi \in C([0,T] \times \Rd)$ is a viscosity sub-solution of
    \begin{equation}
    \label{eq:dual linear}
        \partial_s \psi_s = {\bm E}_{T-s} \nabla \psi_s + \diffc \Delta \psi_s 
    \end{equation}
    if, for every $z_0 = (s_0, x_0) \in [0,T] \times \Rd$ and $U$ neighbourhood of $z_0$ and $\varphi \in C^2(U)$ touching $\psi$ from above (i.e., $\varphi \ge \psi$ on $U$ and $\varphi(z_0) = \psi(z_0)$) we have that
    \begin{equation*}
        \frac{\partial \varphi}{\partial s} (z_0) \le  {\bm E}_{T-s_0} (x_0)  \nabla \varphi (z_0)  + \diffc \Delta \varphi (z_0).
    \end{equation*}
    Conversely, we say that $\psi \in C([0,T] \times \Rd)$ is a super-solution if the inequalities above are reversed for functions touching from below. 
    We say that $\psi \in C([0,T] \times \Rd)$ is a viscosity solution if it is both a viscosity sub and super solution.
\end{definition}
Thus, we introduce the following notion of solution of \eqref{eq:main}:
\begin{definition}
    \label{defn:dual viscosity}
	We say that $(\vec \mu, \{\vec \Psi^T\}_{T\ge 0})$ is an entropy pair
	if:
	\begin{enumerate}
		\item For every $T \ge 0$, 
		$$
		    \vec \Psi^T: \{ \psi_0 : \nabla \psi_0 \in L^\infty (\Rd) \} \longrightarrow 
            \left\{ 
                (\psi^1, \cdots, \psi^n) \in C([0,T] \times \Rd )^n : 
                \frac{\psi^i}{1+|x|} \in L^\infty ((0,T) \times \Rd) 
            \right\}
		$$ 
		is a linear map with the following property: for every $\psi_0$ and $i = 1, \cdots, n$ we have $\psi^{T,i} = \Psi^{T,i}[\psi_0]$ is a viscosity solution of \eqref{eq:dual}.
		\item For each $i = 1, \cdots, n$ and $T \ge 0$, $\mu_T^i \in \mathcal P_1 (\Rd)$ and satisfies the duality condition \eqref{eq:duality}.
	\end{enumerate}
\end{definition}

\begin{remark}
    Notice that in our definition we are not assuming that $\psi_0$ are bounded or decay at infinity, only that they are Lipschitz continuous. This means that they are bounded by $C(1+|x|)$, so they can be integrated against functions in $\mathcal P_1(\Rd)$.
\end{remark}

For convenience, we will simply denote $\vec \Psi = \{ \vec \Psi^T \}_{T\ge 0} $.

\begin{definition}
    We say that $\vec \mu$ is a dual viscosity solution if there exists $\vec \Psi$ so that $(\vec \mu, \vec \Psi)$ is an entropy pair.
\end{definition}

These definitions can be extended to other norms where a duality characterisation exists. The $2$-Wasserstein distance, which is natural for many problems because of the gradient-flow structure, does not have 
any known
such characterisation. However, other spaces like $\dot H^{-1}$ do. We present extension of our results to this setting in \Cref{sec:well-posedness in H-1}.

\begin{remark}
Notice that this duality characterisation is somewhat in the spirit of the Benamou-Brenier formula for the $2$-Wasserstein distance  between two measures, where the optimality conditions involve viscosity solutions of a Hamilton-Jacobi equation.
\end{remark}

\begin{remark}
Notice that we have not requested the time continuity of $\mu$, so the notion of initial trace is not clear. However, it will follow from the continuity of $\psi^T$. Under some assumptions on ${\bm {\mathfrak K}}$ we will show that $\mu \in C([0,T]; \mathcal P_1 (\Rd))$. However, in more general settings, the definition guarantees that the initial trace is satisfied in the weak-$\star$ sense. Notice that if $\psi_0 \in C_c$ and $\psi_t$ is continuous in time then
$$
    \int_\Rd \psi_0 \diff \mu_t = \int_\Rd \psi_t \diff \mu_0 \to \int_\Rd \psi_0 \diff \mu_0.
$$
\end{remark}

\begin{remark}[General linear diffusions]
Our definitions, methods, and results can be extended to the case where $\diffc \Delta u$ is replaced by any other linear operator that commutes with $\partial_x$, and that satisfies the maximum principle, e.g., the fractional Laplacian $\diffc (-\Delta)^s$.
\end{remark}

\subsection{Well-posedness when \texorpdfstring{${\bm K}$}{K} is regularising}
\label{sec:theorem statement K regular}

So far, ${\bm E}_t = -{\bm K}_t [\bm \mu_t]$, i.e., we assume that at every time $t$ the convection comes from a function ${\bm K}_t : \bm \mu \to {\bm E}$. This covers, for example, non-local operators in space. However, our arguments extend to much broader arguments, where ${\bm K}$ could be non-local in time as well. 
We will consider the more general case of
\begin{equation*}
    \bm {\mathfrak K} : C([0,t];X ) \to C([0,t] ; Y ), \qquad \forall t \in [0,T].
\end{equation*}
The results in this section correspond to $X = \mathcal P_1(\Rd)^n$ and $Y$ the space
\begin{equation} 
    \Lipschitz_0 (\Rd; \Rd) =  \left \{ {\bm E} \in C_{loc} (\Rd; \Rd) :   |\nabla {\bm E}| \in L^\infty (\Rd) \right \}
\end{equation} 
However, the scheme of the proof is general and can be extended to other $X,Y$. In \Cref{sec:well-posedness in H-1} we work on the negative homogeneous Sobolev space.

Notice that $\Lipschitz_0 (\Rd; \Rd)$ is very similar $\dot W^{1,\infty}$, but we avoid this terminology to endow it with the following special norm
\begin{equation*}
    \| {\bm E} \|_{\Lipschitz_0} = \| \nabla {\bm E} \|_{L^\infty} + \left\| \frac{{\bm E}}{1 + |x|} \right\|_{L^\infty}.
\end{equation*}
The reason for the $\Lipschitz_0$ nomenclature is that
\begin{equation*}
    |{\bm E}(0)| + \|\nabla {\bm E}\|_{L^\infty} \le \| {\bm E} \|_{\Lipschitz_0} \le 2 (|{\bm E}(0)| + \|\nabla {\bm E}\|_{L^\infty}),
\end{equation*}
so it guarantees the bound
of $\bm E(0)$.
It is not difficult to see that this is a Banach space.
In this setting, we prove the following well-posedness result
\begin{theorem}
    \label{thm:well-posedness in P1}
 Let $\diffc \ge 0$, $T > 0$,
 and assume that for all $t > 0$
        \begin{equation}
            \label{hyp:P1 domain}
            \bm {\mathfrak K}^i : C([0,t]; \mathcal P_1(\Rd)^n) \to C([0,t]; \Lipschitz_0 (\Rd; \Rd)), 
        \end{equation}
        with the property \eqref{eq:continuation},
        maps bounded sets into bounded sets, and satisfies the following Lipschitz condition
        \begin{equation}
            \label{hyp:P1 Lipschitz}
            \sup_{\substack { t \in [0,T]\\ i=1,\cdots,n}} \left\| \frac{\bm {\mathfrak K}^i[\vec \mu]_t - \bm {\mathfrak K}^i[ \widehat {\vec \mu}]_t}{1 +|x|} \right\|_{L^\infty (\Rd)} \le L \, \sup_{\substack { t \in [0,T]\\ i=1,\cdots,n}} d_1 (\mu_t^i, \widehat \mu _t^i ).
        \end{equation}
Then, for each $\vec \mu_0 \in \mathcal P_1 (\Rd)^n$ there exists a unique dual viscosity solution of \eqref{eq:main generalised},
$\mu \in C([0,T]; \mathcal P_1(\Rd)^n)$, and it depends continuously on the initial datum with respect to the $d_1$ distance. This solution $\mu$ also satisfies \eqref{eq:main generalised} in distributional sense.
Furthermore, if $\bm {\mathfrak K}^i[\vec \mu]_t = K^i[\vec \mu_t]$, then the map $S_T : \bm  \mu_0 \in \mathcal P_1 (\Rd)^n \mapsto \bm \mu_T \in \mathcal P_1(\Rd)^n$ is a continuous semigroup.
\end{theorem}

Going back to \eqref{eq:duality characterisation of d1} we have that
\begin{equation*}
    \left| \int_\Rd \psi \diff (\mu - \widehat \mu) \right| \le d_1 (\mu, \widehat \mu) \Lip (\psi), \qquad \forall \psi \text{ such that } \Lip(\psi) < \infty.
\end{equation*}
Thus, we can use the duality relation \eqref{eq:duality} to get the following upper bound
\begin{equation} 
	 \label{eq:duality P1 and weighted Linfty deduction}
	\begin{aligned}
	    d_1 (\mu_T, \widehat \mu_T) &= \sup_{\substack{ \Lip(\psi_0) \le 1 } }  \int_\Rd \psi_0 \diff (\mu_T - \widehat \mu_T) \\
	    &= \sup_{\Lip(\psi_0) \le 1} \left(  \int_\Rd \psi_T^T \diff \mu_0 -  \int_\Rd \widehat \psi_T^T \diff \widehat \mu_0 \right) \\
	    &= \sup_{\Lip(\psi_0) \le 1} \left( \int_\Rd  \psi_T^T \diff (\mu_0 - \widehat \mu_0) +  \int_\Rd ( \psi_T^T - \widehat \psi_T^T ) \diff \widehat \mu_0  \right) \\
	    &\le d_1(\mu_0 , \widehat \mu_0) \sup_{\Lip(\psi_0) \le 1} \Lip(\psi_T^T)  + \sup_{\Lip(\psi_0) \le 1} \int_\Rd ( \psi_T^T - \widehat \psi_T^T ) \diff \widehat \mu_0.  
	\end{aligned}
\end{equation} 
To exploit this fact we take advantage of the following, rather obvious, estimate
\begin{align*}
    \left| \int \psi(x) \diff \mu (x) \right| & \le  \left\|  \frac{\psi}{1+|x|} \right \|_{L^\infty} \left( 1 + \int_\Rd |x| \diff \mu(x) \right). 
\end{align*}
In fact, we show that, on the second supremum in \eqref{eq:duality P1 and weighted Linfty deduction}, we can focus on functions such that $\psi_0(0) = 0$. The reason for this election will be made clearer later.
\begin{lemma}
\label{lem:continuous dependence}
Let $\mu$ and $\widehat \mu$ two dual viscosity solutions of \eqref{eq:main E known} such that ${\bm E}, \widehat {\bm E} \in C([0,T];\Lip_0 (\Rd, \Rd))$. Then, letting $\psi^T$ and $\widehat \psi^T$ be the corresponding solutions of the dual problems \eqref{eq:dual E known}, we have that
\begin{equation}
    \label{eq:duality P1 and weighted Linfty}
    d_1 (\mu_T, \widehat \mu_T) \le  d_1(\mu_0 , \widehat \mu_0) \sup_{\substack{\Lip(\psi_0) \le 1  }} \Lip(\psi_T^T) + \int_\Rd (1 + |x| ) \diff \widehat \mu_0(x)   \sup_{ \substack{ \Lip(\psi_0) \le 1 \\ \psi_0(0) = 0 }} \left\|  \frac{\psi_T^T - \widehat \psi_T^T}{{1+|x|}} \right \|_{L^\infty} .
\end{equation} 
\end{lemma}
\begin{proof}
Notice that $\psi^T = \Psi^T[\psi_0]$, by linearity 
$$ \Psi^T[\psi_0] = \Psi^T[\psi_0 - \psi_0(0)] + \psi_0(0) \Psi^T [1] $$
By the uniqueness in \Cref{prop:dual existence Lipschitz}, which we will prove below,
since the problem does not contain a zero-order term $\Psi^T [1] = 1 =\widehat \Psi^T[1]$. Therefore, we can write
$$ \Psi^T[\psi_0] - \widehat \Psi^T[\psi_0] = \Psi^T[\psi_0 - \psi_0(0)] - \widehat \Psi^T[\psi_0 - \psi_0(0)].$$
Thus, we can take the second supremum of \eqref{eq:duality P1 and weighted Linfty deduction} exclusively on functions such that $\psi_0(0) = 0$.
\end{proof}

\begin{remark}
    We recall a basic property of the 1-Wasserstein distance. It turns out that the first moment is precisely the distance to the Dirac delta, i.e.,
    $$
            \int_\Rd  |x| \diff \mu(x) = d_1 (\mu, \delta_0).
    $$
    Hence, our fixed point argument will be performed in the balls
    \begin{equation}
    	\label{eq:ball in Wasserstein}
        B_{\mathcal P_1} (R) = \left \{ \mu \in \mathcal P_1 (\Rd)  :  \int_\Rd  |x| \diff \mu(x)  \le R \right\}.
    \end{equation} 
    This explains that it is natural to work with the metric of $(1 + |x|) L^\infty$.
\end{remark}

We will devote \Cref{sec:solutions of PE*} to the estimates of $\psi$ that allow us to control each term of \eqref{eq:duality P1 and weighted Linfty}.  The estimates are recovered by standard, albeit involved, arguments.

The first term in \eqref{eq:duality P1 and weighted Linfty}, which corresponds to continuous dependence respect to the initial datum, requires only the estimation of a coefficient. If $\bm {\mathfrak K}$ is well-behaved, we expect to obtain an estimate of the type
\begin{equation*}
    \sup_{\Lip(\psi_0) \le 1} \Lip(\psi_T^T) \le C(T, \bm {\mathfrak K}, \mu_0 ).
\end{equation*}
For the well-posedness theorem $\widehat \mu_0 = \mu_0$, we can avoid the first term in \eqref{eq:duality P1 and weighted Linfty} and this constant is not relevant.

The second term in \eqref{eq:duality P1 and weighted Linfty}, which represents continuous dependence on ${\bm E}$ for the dual problem, is harder to estimate.  
We prove first that
even if $\psi_0 \notin L^\infty $ but is Lipschitz and ${\bm E}, \widehat {\bm E}$ are bounded, then $\psi_s^T - \widehat \psi_s^T \in L^\infty(\Rd)$.
A second, more elaborate, argument will allow us to work of the case where ${\bm E}$ grows no-more-than linearly at infinity.
If \eqref{eq:dual E known} has good continuous dependence on ${\bm E}$, we can expect
\begin{equation*}
    \left\|  \frac{\psi_T^T - \widehat \psi_T^T}{{1+|x|}} \right \|_{L^\infty}  \le C \omega(T) \| \bm {\mathfrak K}[ \mu ] - \bm {\mathfrak K} [\widehat \mu ]\|,
\end{equation*}
where we leave the second norm unspecified for now. 
So, if $\bm {\mathfrak K}$ is Lipschitz in this suitable norm, then we will be able to use Banach's fixed point theorem.

\begin{remark}
\cite{Karlsen2003} prove that if $\rho_0, \widehat \rho_0 \in L^1 \cap L^\infty \cap BV$ and $\rho, \widehat \rho \in L^\infty([0,T]; L^1)$ are entropy solutions of
\begin{equation*}
    \partial_t \rho_t + \diver(f(\rho) {\bm E}(x)) = \Delta A (\rho), \qquad  \partial_t \widehat \rho + \diver(\widehat f(\widehat\rho) \widehat {\bm E}(x)) = \Delta A(\widehat\rho)  
\end{equation*}
where ${\bm E}, \widehat {\bm E} : \Rd \to \Rd$, and we have an a priori estimates $\rho, \widehat \rho : \Rd \to I \subset \R $ with $I$ bounded and
\begin{equation*}
    [\rho_t]_{BV(\Rd)} ,[\widehat \rho_t]_{BV(\Rd)} \le C_{\rho} , \qquad \forall t \in (0,T).
\end{equation*}
Then, 
\begin{align*}
    \| \rho_t - \widehat \rho_t \|_{L^1} \le \| \rho_0 - \widehat \rho_0 \|_{L^1} + C t \Big( \| {\bm E} - \widehat {\bm E} \|_{L^\infty} + [ {\bm E} - \widehat {\bm E} ]_{BV} + \| f - \widehat f \|_{W^{1,\infty}(I;\R)} \Big)
\end{align*}
where $\rho, \widehat \rho : \Rd \to I \subset \R $ and $C$ depends on $C_\rho$ and the norms of $f,\widehat f \in W^{1,\infty}$, ${\bm E}, \widehat {\bm E} \in L^\infty \cap BV$. 
In \cite{DiFrancesco2019} the authors use the fixed-point argument of entropy solutions only to prove uniqueness. Existence is done by proving the convergence of the particle systems, which is their main aim. They use \cite{Karlsen2003} which deals directly with the conservation law in $L^1$, we obtain simple estimates for the Hamilton-Jacobi dual problem in $L^\infty$ using the maximum principle.
Our duality argument is not directly extensible to $L^1 (\Rd)$.
Applying the duality characterisation of $L^1$ and \eqref{eq:duality} we can compute
\begin{align*}
   \| \mu_T - \widehat \mu_T \|_{L^1} &= \sup_{|\psi_0| \le 1}  \int_\Rd \psi_0 \diff (\mu - \widehat \mu) 
   \le \sup_{|\psi_0| \le 1} \left( \int_\Rd  \psi_T^T \diff (\mu_0 - \widehat \mu_0) +  \int_\Rd ( \psi_T^T - \widehat \psi_T^T ) \diff \widehat \mu_0  \right) 
\end{align*}
Due to the maximum principle we will show that $|\psi_s^T| \le 1$ and hence
\begin{equation*}
    \| \mu_T - \widehat \mu_T \|_{L^1} \le \| \mu_0 - \widehat \mu_0 \|_{L^1} + \| \widehat  \mu_0 \|_{L^1} \sup_{|\psi_0| \le 1} |\psi_T^T - \widehat \psi_T^T|.
\end{equation*}
There is no apparent way to bound the second supremum with a quantity related only to   $|{\bm E} - \widehat {\bm E}|$. A quantity depending on $|\nabla \psi_0|$ appears, and there is no available bound of its supremum over functions such that $|\psi_0| \le 1$.
\end{remark}

\subsection{Stability theorem}
\label{sec:thm stability}
We conclude the main results by stating a stability theorem. This allows to prove existence of solutions by approximation, in more general settings that the well-posedness theory. We will use it below for several applications.
\begin{theorem}
\label{thm:stability}
Let 
$(\vec \mu^k , \vec \Psi^k) $
be a sequence of entropy pairs in the sense of \Cref{defn:dual viscosity} corresponding to some operators $\vec {\bm {\mathfrak K}}^k$ under the assumptions of \Cref{thm:well-posedness in P1}.
Assume that
\begin{enumerate}
    \item $\diffc^k \to \diffc^\infty$.
    \item For every $t> 0$, $\vec \mu_t^k \overset \star \rightharpoonup \vec \mu_t^\infty$ in $\mathcal M (\Rd)$.
    \item For every $\psi_0 \in C_c^\infty (\Rd)$ and $T > 0$ we have $\vec \Psi^{T,k} [\psi_0] \to \vec \Psi^{T,\infty}[ \psi_0 ]$ uniformly over compacts of $[0,T] \times \Rd$.
    This allows to pass to the limit in estimate \eqref{eq:dual psi/v bound} below.
    \item $\vec {\bm {\mathfrak K}}^k [\vec \mu^k] \to \vec {\bm {\mathfrak K}}^\infty[\vec \mu^\infty]$ uniformly over compacts  of $[0,T] \times \Rd$.
    \item For every $\psi_0 \in C_c^\infty (\Rd)$ we have that
    \begin{equation*}
        \int_\Rd \vec \Psi^{T,k}[\psi_0]_T \diff \vec \mu_0^k \to \int_\Rd \vec \Psi^{T,\infty}[\psi_0]_T \diff \vec \mu_0^\infty.
    \end{equation*}
\end{enumerate}
Then $(\vec \mu^\infty, \vec  \Psi^\infty)$ is an entropy pair of the limit problem.
\end{theorem}

The result follows directly from the definition and the following classical result of stability of viscosity solutions (see, e.g., \cite{crandall+ishii+lions1992users-guide-viscosity}).
\begin{theorem}
\label{lem:viscosity solutions stability}
    Let $\diffc^k \ge 0$, $\psi^k \in C([0,T] \times \Rd)$, ${\bm E}^k \in C([0,T]\times \Rd)^d$ be viscosity solutions of
    \begin{equation*}
        \partial_s \psi_s^k = {\bm E}_{T-s}^k \nabla \psi_s^k + \diffc^k \Delta \psi_s^k
    \end{equation*}
    and assume that
     $\diffc^k \to \diffc^\infty$ 
    and $\psi^k \to \psi^\infty$,   ${\bm E}^k \to {\bm E}^\infty$ uniformly over compacts of $[0,T] \times \Rd $.
    Then, $\psi^\infty$ is a viscosity solution of
    \begin{equation*}
        \partial_s \psi_s^\infty = {\bm E}_{T-s}^\infty \nabla \psi_s^\infty + \diffc^\infty \Delta \psi_s^k.
    \end{equation*}
\end{theorem}

\subsection{The limits of the theory: the diffusive Newtonian potential in \texorpdfstring{$d=1$}{d=1}}
\label{sec:Newtonian}

Let us consider the case $n = d = 1$ with $K$ coming from the gradient-flow structure \eqref{eq:gradient flow of convolution} and the Newtonian potential
$W(x) = -|x|$. Then $\nabla W(x) = \sign(x)$. 
This $W$ falls outside the theory developed in \Cref{thm:well-posedness in P1}, but can be approximated by solutions in this framework.
It was proved in \cite{Bonaschi2015}
that for $\mu_0 = \delta_0$ then the gradient flow solution is not given by particles but rather
\begin{equation*}
    \mu_t = \frac{1}{2t} \chi_{[-t,t]}.
\end{equation*}
Naturally, any smooth approximation of $W$ (for example $W_k = -(x^2 + \frac 1 k)^{\frac 1 2}$) yields $\mu^k_t = \delta_0$. 
However, solving with $\mu_0^m = \frac{1}{2m} \chi_{[-m,m]}$ gives the expected reasonable solution. We therefore have the following diagram
\begin{equation*}
    \xymatrix@C=.25em{
    \mu^{k,m} \ar[rrrrr]^{m\to 0} \ar[d]^{k\to \infty} & & & & & \mu^k  & = & \delta_0 \ar[d]^{k\to \infty} \\
    \mu^{m} \ar[rrrrr]^{m\to 0} &&&&& \mu  & \ne  & \delta_0.
    }
\end{equation*}

Clearly, $\mu_t^k \to \delta_0$ in any conceivable topology. Checking whether $\delta_0$ is or is not a distributional solution for $W $ is not totally trivial. Notice that $\sign$ (as a distributional derivative of $|x|$) is not pointwise defined at $0$, so the meaning of $\int \varphi \sign * \delta_0 \diff \delta_0$ is not completely clear. With the choice $\sign(0) = 0$ one could be convinced that $\mu_t = \delta_0$ is a distributional solution. But this seems arbitrary.

In fact, if one takes the approximating entropy pairs $(\mu^k, \Psi^k)$, it is not difficult to show that a discontinuity appears in $\psi_s^{k,T}$ as $k \to \infty$. Hence, we do not have uniform convergence of $\psi^{k,T}$. 
To show this fact, let $E^k = -\partial_x (W_k * \mu^k)$. Since $\mu^k = \delta_0$, we have that 
$$
    E^k _t (x) = -\partial_x W_k = \frac{x}{(x^2 + \frac 1 k)^{\frac 1 2}}.
$$
Since $E^k$ does not depend on $t$, $\psi^{k,T}$ does not depend on $T$. Therefore, we drop the $T$ to simplify the notation. We have to solve the equation
\begin{equation*}
	\partial_t \psi^k_t (y) = \frac{y}{(y^2 + \frac 1 k)^{\frac 1 2}} \partial_y \psi^k_t.
\end{equation*}
This equation can be solved by characteristics $\psi^k_t (Y^k_t(y_0)) = \psi_0 (y_0)$ given by
\begin{equation*}
	\partial_t Y^k_t = - \frac{Y^k_t}{(Y_t^2 + \frac 1 k)^{\frac 1 2}}
\end{equation*}
As $k \to \infty$ we recover 
$
	Y^{k} _t (y_0) \to  Y_t(y_0) = y_0 - \sign(y_0) t
$ or, equivalently, that
\begin{equation*}
	\psi^k_t (y) \to \psi_t (y) = \psi_0 (  y + \sign (y) t   ).
\end{equation*}
Therefore, $\psi_t(0^-) = \psi_0 (-t)$ and $\psi_t(0^+) = \psi_0 (t)$. Thus, $\psi_t$ is not continuous in general. 
So we cannot pass to the limit $\mu^k \to \mu$ in terms of entropy pairs.

The gradient flow solution is recovered by the entropy pairs if one approximates $\delta_0$ by $\mu_0^m = \frac{1}{2m} \chi_{[-m,m]}$ and passes to the limit. It is also not difficult to show that the entropy pairs of twice regularised problem,  $(\mu^{k,m}, \Psi^{k,m})$ do converge in the sense of entropy pairs as $k \to \infty$ to $(\mu^m, \Psi^m)$. And then we can also pass to the limit in the sense of entropy pairs as $m \to \infty$. 

Thus, the stable semigroup solutions are entropy pairs, and this notion of solution is powerful enough to discard the wrong order of the limits. 

\subsection{Open problems}
\label{sec:open problem}
We finally briefly discuss the key difficulties for some related problems that our duality theory cannot cover.

\paragraph{Nonlinear diffusion.} 
We could deal, in general, with problems of the form
\begin{equation}
\label{eq:main nonlinear diffusion}
    \partial_t \mu_t^i = \diver (\mu^i {\bm {\mathfrak K}}^i [\vec \mu] ) + \diver (\vec M^i_t[ \mu^i _t] \nabla \mu_t^i ).
\end{equation}
Going back to \eqref{eq:main nonlinear diffusion}, the dual problem can be written in the form
\begin{equation*}
	\frac{\partial \psi_s}{\partial s}  =  {\bm E}_{T-s}  \nabla \psi_s + \diver ( \vec A_s \nabla \psi_s  ) .
\end{equation*}
We have not been able to estimates of $\| \nabla ( \psi - \overline \psi) \|_{L^\infty}$ of $\| \psi - \overline \psi \|_{H^1}$ with respect to $E - \overline E$ and $\vec A - \overline {\vec A}$ using the techniques below. The main difficulty is that letting $v = \frac{\partial \psi}{\partial x_i}$ and $\overline v = \frac{\partial \overline \psi}{\partial x_i}$, the reasonable extension of \eqref{eq:dual E known eq for v - hat v} contains a term $-\diver ( (\vec A - \overline {\vec A}) \nabla \overline v) $. Controlling these terms requires estimates of $D^2 \psi$, which are not in principle present for the initial data we discuss. However, smarter estimates and choices of space for $\mu$ and $\psi$ may lead to a successful extension of our results.

\paragraph{Coupling in the dual problem.} 
Another possible extension is the analysis of more general systems of the form
\begin{equation}
    \frac{\partial \mu^i}{\partial t} = \sum_{j=1}^n \diver \left( \mu^j {\bm {\mathfrak K}}^{ij} [\vec  \mu]\right) = \sum_{j,k=1}^n \frac{\partial}{\partial x_k} \left( \mu^j {\bm {\mathfrak K}}^{ijk}\right), \qquad i = 1, \cdots, n.
\end{equation}
The dual problem is now not decoupled, and we have to solve it a system
\begin{equation*}
	\frac{\partial }{\partial t} \vec \psi_s^T  =  \nabla \vec \psi_s^T \cdot \vec {\bm {\mathfrak K}}^* [\vec \mu_{T-s}] .
\end{equation*}
This kind of problems does not have a comparison principle in general, for example
\begin{equation*}
\begin{dcases} 
    \partial_t u = \partial_x v, \\
    \partial_t v = \partial_x u
\end{dcases}
\end{equation*}
leads to the wave equation $\partial_{tt} u =\partial_{t} (\partial_x v) = \partial_x \partial_t v = \partial_{xx} u$.
Some examples of equations resembling this structure appear in mathematical biology models, e.g., \cite{S.K.T1979}.

\paragraph{Problems with saturation.}
Some authors have studied the case
${\bm {\mathfrak K}}_i [\vec \mu]  = \frac{\theta(\mu)}{\mu} H_i [\vec \mu] $
where $\theta (0) = \theta(M) = 0$. 
This problem was studied for $n = 1$ and $D = 0$ by \cite{DiFrancesco2019} using a fixed point argument with the estimates by \cite{Karlsen2003}.
In this setting the solutions with initial datum $0 \le \rho_0 \le M$ remain bounded. Therefore, we expect bounded (even continuous) solutions $\rho$. So the natural duality would be $\psi$ integrable. We point the reader to \cite{Endal2018}, where the authors study the duality between $L^\infty$ estimates for entropy solutions of conservation laws and $L^1$ bounds for viscosity solutions of Hamilton-Jacobi equations.

\section{Viscosity solutions of \texorpdfstring{\eqref{eq:dual E known}}{(PE*)}}
\label{sec:solutions of PE*}

The aim of this section is to prove the following existence, uniqueness, regularity and continuous dependence result for \eqref{eq:dual E known} which is one of the key tools in this paper. The usefulness of the estimates obtained is clear going back to \Cref{lem:continuous dependence}.
\begin{proposition}
\label{prop:dual existence Lipschitz}
     Let $D \ge 0$, $\bm E \in C([0,T]; \Lip_0 (\Rd)) $ and $\psi_0 $ be such that $ \nabla \psi_0 \in L^\infty (\Rd) $. Then, there exists a unique viscosity solution of \eqref{eq:dual E known} such that $\frac{ \psi }{ 1 + |x| } \in L^\infty ((0,T)\times\Rd)$. Furthermore, it satisfies the following estimates
    \begin{align}
        \label{eq:dual Linfty estimate}
	\|\psi_s \|_{L^\infty} &\le \| \psi_0 \|_{L^\infty}, \\
	\label{eq:dual grad estimate}
		\vvvert \nabla \psi_s \vvvert_{L^\infty} &\le  \vvvert \nabla \psi_0 \vvvert_{L^\infty}  \exp \left( C { \int_0^T  \left\| \nabla \vec E_{T-s }\right\|_{L^\infty} \diff \sigma  } \right)\\
    \label{eq:dual psi/v bound}
        \left \|\frac{\psi_s}{1 + |x|} \right\|_{L^\infty} &\le C \left \|\frac{\psi_0}{1 + |x|} \right\|_{L^\infty} \exp \left( \diffc T + \int_0^T   \left\| \frac{ \bm E_{T-\sigma} }{1 + |x|} \right\|_{L^\infty}
         \diff \sigma \right), 
    \end{align}
    where $C = C(d)$ depends only on the dimension, 
    and $\vvvert \nabla \psi_s \vvvert_{L^\infty} = \sup_i \| \tfrac{\partial \psi}{\partial x_i}\|_{L^\infty}$. 
    If $\psi_0 \ge 0$ then $\psi \ge 0$.
    Moreover, we obtain the following time-regularity estimate
    \begin{equation}
    	\label{eq:dual time estimate}
    \begin{aligned}
         \left \|\frac{\psi_{s+h} - \psi_s}{1 + |x|} \right\|_{L^\infty} &\le C \Bigg ( h + D^{\frac 1 2} ( h^{\frac 1 2}  +  h^{\frac 3 2})+ \sup_{\sigma \in [0,s]}  \left\| \frac{ \bm E_{T-(\sigma+h)} - \bm E_{T-\sigma}}{1 + |x|} \right\|_{L^\infty} \Bigg) ,
    \end{aligned}
    \end{equation}
    where $ C = C(T,D, \left\| \psi_0 \right\|_{Lip_0} , \sup_\sigma \| \bm E_\sigma  \|_{Lip_0} )$.
    Lastly, given $\widehat \psi_0 = \psi_0$ and $\widehat{\bm E}$ in the same hypotheses above, there exists a corresponding solution of \eqref{eq:dual linear}, denoted by $\widehat \psi$, and we have the continuous dependence estimate
    \begin{equation}
        \label{eq:dual continuous dependence on E - hat E / (1 + |x|)}
        \left \| \frac{\psi_s - \widehat \psi_s}{1 + |x|}  \right\|_{L^\infty} \le C \int_0^T  \left\| \frac{E_{T-\sigma} - \widehat{\bm E}_{T-\sigma}}{1+|x|}  \right\|_{L^\infty} \diff \sigma 
    \end{equation}
    where $C = C\left( d, \diffc,  \|  \psi_0 \|_{Lip_0}, 
    \int_0^T \left\|   E_{T-\sigma}   \right\|_{Lip_0} \right)$ is monotone non-decreasing in each entry. 
\end{proposition}
We will prove this result in \Cref{sec:proof of prop:dual prop:dual existence Lipschitz}.
The uniqueness of viscosity solutions leads to the following
\begin{corollary}
  Assume that $\vec {\bm {\mathfrak K}}: C([0,T]; \mathcal P_1 (\Rd)) \to \mathcal C([0,T]; \Lipschitz_0(\Rd; \Rd))$, and let $(\mu, \Psi)$ and $(\mu, \widehat \Psi)$ be entropy pairs of \eqref{eq:main}, then $\Psi = \widehat \Psi$.
\end{corollary}

There are some immediate consequences that come from the linearity of equation.
\begin{remark}
Notice that the properties above imply a comparison principle. Given two solutions, we have that $\psi_0 \le \widehat \psi_0$, then $\widetilde { \psi}_0 = \widehat \psi_0 - \psi_0 \ge 0$. Then, by linearity, uniqueness, and preservation of positivity $\widehat \psi - \psi = \widetilde { \psi} \ge 0$.
\end{remark}

Notice that the generic initial datum $\Lip(\psi_0) \le 1$, need not be bounded. This can initially seem like a problem for uniqueness, since we cannot prescribe conditions at infinity. However, following Aronson \cite[Theorem 2 and 3]{Aronson1968}, existence and uniqueness are obtained under the assumption that $e^{-\lambda |x|^2} u_t \in L^2 ((0,T) \times \Rd)$ if the coefficients are bounded. We tackle this issue by studying weighted versions of our solutions
$    v_s(x) = \psi_s(x) / \eta(x)
$, which solve the following problem
\begin{equation}
\label{eq:dual problem E known v}
    \partial_s v = v \frac { E_{T-s} \cdot \nabla \eta + \diffc\Delta \eta  }{\eta}  + \nabla v \cdot \frac{E_{T-s} +2 \diffc \nabla \eta}{\eta} + \diffc \Delta v.
\end{equation}
Notice that, if $\eta (x) = (1 +  |x|^2)^{k/2}$ with $k \ge 1$, then $\nabla \eta / \eta \sim |x|^{-1}$ at infinity. If $E$ is Lipschitz, then all the coefficients of the equation above are bounded.
And, if $k > 1$, then $|v_0(x)|\le C(1+|x|)^{1-k} \to 0 $ as $|x|\to\infty$.

\begin{remark}
A similar argument can be adapted to initial data which are weighted with respect to $1 + |x|^p$ for any $p \ge 1$, if this is satisfied by the initial datum. However, this escapes the interest of this work.
\end{remark}

\subsection{A priori estimates in \texorpdfstring{$L^\infty$}{Linfty}}

\subsubsection{General linear problem}

For the dual problem we are interested in the existence and uniqueness of the linear parabolic problem in non-divergence form
\begin{equation}
\label{eq:GLP}
\begin{dcases} 
	\frac{\partial u_s}{\partial s} = f_s +  a_s u_s + \bm b_s \cdot \nabla u_s + \diffc \Delta u_s & \text{for all } s > 0, x \in \Rd \\
	u_s \to 0 & \text{as } |x| \to \infty \text{ for all } s > 0.
\end{dcases} 
\end{equation}
The theory of existence and uniqueness of classical solutions 
dates back to \cite{ladyzhenskaia1968parabolic}. When the coefficients are smooth, the linear problem can be rewritten in divergence form as
\begin{equation}
\label{eq:GLP divergence}
	\partial_s u_s = f_s +  (a_s - \diver \bm b_s) u_s + \diver ( u_s \bm b_s + \diffc \nabla u_s ) .
\end{equation}

We focus on obtaining suitable a priori estimates assuming that the data and solutions are smooth enough, and these estimates pass to the limit to the unique viscosity solution by approximation of the coefficients. 

By studying \eqref{eq:GLP} at the point of maximum/minimum, we formally have that
\begin{equation*}
    \frac{d}{ds} \|u_s\|_{L^\infty} \le  \| f_s \|_{L^\infty} +  \|a_s\|_{L^\infty} \| u_s \|_{L^\infty} .
\end{equation*}
This intuition can be made precise by the following result 
\begin{lemma}
    \label{lem:GLP L infty estimate}
	If $a,f \in C(\Rd)$ and bounded, $\bm b \in C(\Rd)^d$, and 
	$u$ is a classical solution of \eqref{eq:GLP}
    {such that $\sup_{(0,T) \times \partial B_R} |u| \to 0$ as $R \to \infty$ for each $T$}%
    , then
\begin{equation}
\label{eq:linear viscosity universal bound}
   \|u_s\|_{L^\infty} \le \| u_0 \|_{L^\infty} \exp \left( \int_0^s \|a_\sigma \|_{L^\infty} \diff \sigma \right) + \int_0^s \|f_\sigma\|_{L^\infty} \exp \left( \int_\sigma^s \|a_\kappa\|_{L^\infty} \diff \kappa  \right) \diff \sigma .
\end{equation}
\end{lemma}

\begin{proof}
    
	Define
	$$ 		\overline u_s (x) = \| u_0 \|_{L^\infty} \exp \left( \int_0^s \|a_\sigma \|_{L^\infty} \diff \sigma \right) + \int_0^s \|f_\sigma\|_{L^\infty} \exp \left( \int_\sigma^s \|a_\kappa\|_{L^\infty} \diff \kappa  \right) \diff \sigma .$$ 
	This is a classical super-solution of \eqref{eq:GLP} and 
    $u_0 \le \overline u_0$. 
    Using the weak maximum principle on a ball $B_R$ and the initial condition we get
    \[ 
        \min_{(0,T) \times B_R} (\overline u - u) = \min_{\substack{ \{0\} \times B_R \\ \cup (0,T) \times \partial B_R}} (\overline u - u) \ge \min \left\{ 0, \min_{(0,T) \times \partial B_R} ( \overline u - u) \right\} .
    \] 
    As $R \to \infty$ this tends to $0$ by hypothesis. 
    We conclude that $u \le \overline u$.
    Similarly, we prove that $-u \le \overline u$.
\end{proof}

Now we compute estimates on the solution formally assuming that $\psi \in C ([0,T] ; C_0^1( \Rd) )$,
and that $\bm E$ is smooth and satisfies all necessary bounds.
Later, we will justify this formal computations.

\subsubsection{\texorpdfstring{$L^\infty$ estimates of $\psi^T$, $\nabla \psi^T$ and $\psi^T/(1+|x|)$}{Linfty estimates of psiT}}
Going to \eqref{eq:dual E known} 
{%
when we have suitable decay of $\psi$ we deduce \eqref{eq:dual Linfty estimate} by \Cref{lem:GLP L infty estimate}.
}%
When we define
$
	U^i_s = \frac{\partial \psi_s^T} {\partial x_i } ,
$
we recover that
\begin{align*}
	\frac{\partial U^i_s}{\partial s} 
	&=  \nabla \psi_s \cdot  \frac {\partial \bm E_{T-s}}{\partial x_i} + \nabla U_s^i \cdot \bm E_{T-s}  + \diffc \Delta U_s^i  \\
	&= \sum_{j=1}^d U^j_s \frac {\partial E^j_{T-s}}{\partial x_i} + \nabla U_s^i \cdot \bm E_{T-s}  + \diffc \Delta U_s^i  .
\end{align*}
Since this is a system, we cannot directly apply  \Cref{lem:GLP L infty estimate}. Nevertheless, we compute
\begin{align*}
	\frac 1 p \frac{\diff}{\diff s} \int_\Rd |U_s^i|^p 
	&= 
		\int_\Rd |U_s^i|^{p-2} U_s^i  \left(  \sum_{j=1}^d U^j_s \frac {\partial E^j_{T-s}}{\partial x_i} + \nabla U_s^i \cdot \bm E_{T-s}  + \diffc \Delta U   \right)  \\ 
	&= 	\sum_{j=1}^d  \int |U_s^i|^{p-2} U_s^i   U^j_s \frac {\partial E^j_{T-s}}{\partial x_i}  
	+ \frac 1 p \int \nabla |U_s^i|^p \cdot \bm E_{T-s}  
	+  \diffc \int |U_s^i|^{p-2} U_s^i \Delta U_s^i \\ 
	&= 	\| U_s^i \|_{L^p}^{p-1} \sum_{j=1}^d     \| U_s^j \|_{L^p} \left \|\frac {\partial E^j_{T-s}}{\partial x_i}  \right\|_{L^\infty }
		- \frac 1 p \int_\Rd  |U_s^i|^p \diver  \bm E_{T-s}  
	\\
	&\qquad 	-  (p-1) \diffc \int_\Rd |U_s^i|^{p-1} |\nabla U_s^i|^2 \\
	&\le 	\| U_s^i \|_{L^p}^{p-1} \sum_{j=1}^d     \| U_s^j \|_{L^p} \left \|\frac {\partial E^j_{T-s}}{\partial x_i}  \right\|_{L^\infty } +  \frac 1 p   \| U_s^i \|_{L^p}^p \left \| \diver  \bm E_{T-s}  \right\|_{L^\infty}
\end{align*}
Using the norm equivalence of norms of $\Rd$, we have
\begin{equation*}
	\frac{\diff}{\diff s} \sum_{i=1}^d \| U_s^i \|_{L^p}^p 
	\le \left(  
		\sum_{i=1}^d  \| U_s^i \|_{L^p}^p \right) \left(p C(d)  \sup_i \sum_{j=1}^d \left \|\frac {\partial E^j_{T-s}}{\partial x_i}  \right\|_{L^\infty }  
		+ \left \| \diver  \bm E_{T-s}  \right\|_{L^\infty} 
	\right) 
\end{equation*}
Eventually, we have that
\begin{equation}
	\label{eq:psi derivative Lp estimate}
	\left( \sum_{i=1}^d \| U^i_s \|_{L^p}^p \right) ^{\frac 1  p} 
		 \le 
		 \left(  \sum_{i=1}^d  \| U^i _0\|_{L^p}^p \right) ^{\frac 1 p}
		 \exp \left(
		 	C(d)  \int_0^s \sup_i \sum_{j=1}^d \left \|\frac {\partial E^j_{T-\sigma}}{\partial x_i}  \right\|_{L^\infty }  \diff \sigma  
		 	+ \frac 1 p \int_0^s \left \| \diver  \bm E_{T-\sigma }  \right\|_{L^\infty}  \diff \sigma 
		 \right) .
\end{equation}
As $p \to \infty$, we obtain \eqref{eq:dual grad estimate}.
Take $\eta(x) = (1 + |x|^2)^{\frac 1 2}$. Recalling \eqref{eq:dual problem E known v} and applying \Cref{lem:GLP L infty estimate}, we deduce
\begin{align*}
    \left \|\frac{\psi_T^T}{1 + |x|} \right\|_{L^\infty} \le C \left \|\frac{\psi_0}{1 + |x|} \right\|_{L^\infty} \exp \left(  \diffc \| \Delta \eta \|_{L^\infty} T + \int_0^T  \left\| \frac {\bm E_{T-\sigma}} {\eta} \right\|_{L^\infty}  
     \diff \sigma \right) .
\end{align*}
Since $\eta$ is regular we eventually deduce \eqref{eq:dual psi/v bound} where $C$ is such that $C^{-1} \le \frac{\eta(x)}{1+|x|} \le C$, and it does not depend even on the dimension.
For the continuous dependence, we write (dropping $s$ and $T-s$ from the subindex for convenience)
\begin{equation}
\label{eq:dual E known eq for v - hat v}
\begin{aligned}
    \partial_s (v - \widehat v) =& \,v \frac{ \bm E - \widehat {\bm E}}{\eta} \cdot \nabla \eta   + \nabla v \cdot \frac{\bm E - \widehat {\bm E} }{\eta}  \\
    & + ({v - \widehat v}) \frac{  \widehat v \cdot \nabla \eta + \diffc\Delta \eta  } {\eta}   + \nabla (v - \widehat v) \cdot \frac{\widehat{\bm E} +2 \diffc \nabla \eta}{\eta} + \diffc \Delta (v - \widehat v) .
\end{aligned} 
\end{equation}
Then we have
\begin{align*}
    |a_s| &= \left| \frac{ \widehat v \cdot \nabla \eta }{\eta}  \right| \le C(1 + |\widehat v|) 
\end{align*}
and
\begin{align*}
    |f_s| &= \left|\frac{v}{\eta} \left( ( \bm E - \widehat{\bm E}) \cdot \nabla \eta  \right)  + \nabla v \cdot \frac{\bm E - \widehat{\bm E} }{\eta} \right| 
    \le \left| \frac{ E- \widehat{\bm E}} \eta \right | (|v| + |\nabla v| ).
\end{align*}
Notice that
\begin{equation*}
    \nabla v = \nabla \frac \varphi {\eta} = \frac{1}{\eta^2} \left( \eta \nabla \psi - \psi \nabla \eta \right) = \frac{1}{\eta} \nabla \psi + \frac{v}{\eta} \nabla \eta.
\end{equation*}
Hence, we can deduce using \Cref{lem:GLP L infty estimate} that
\begin{equation}
\label{eq:dual E known v continuous dependence general}
    \| v_s - \widehat v_s \|_{L^\infty (\Rd)} \le C_1 \| v_0 - \widehat v_0 \|_{L^\infty} + C_2 \int_0^T  \left\| \frac{\bm E_{T-\sigma} - \widehat{\bm E}_{T-\sigma}}{1+|x|}  \right\|_{L^\infty} \diff \sigma
\end{equation}
where
\begin{equation}
    \label{eq:dual continuous dependence on E - hat E / (1 + |x|) b}
    \begin{aligned}
         C_1 = C\left(T,D,\left\| \psi / (1 + |x|)\right\|_{L^\infty}, \| \bm E \|_{L^\infty} \right ), \qquad
    C_2 &= \sup_{s \in [0,T] } \left(  \| v_s \|_{L^\infty} \| \nabla \eta \|_{L^\infty} + \| \nabla v_s \|_{L^\infty} \right), 
    \end{aligned} 
\end{equation}
which can be estimated using \eqref{eq:dual grad estimate} and \eqref{eq:dual psi/v bound}.

Eventually, since $\psi_0^T - \widehat \psi_0^T = 0$, we recover \eqref{eq:dual continuous dependence on E - hat E / (1 + |x|)}.
Notice that $C_1$ cannot be uniformly bounded over the set $\Lip(\psi_0) \le 1$, where we can bound $C_2$.
Here is where the assumptions  
that ${\bm {\mathfrak K}}_i$ satisfies \eqref{hyp:P1 domain} and \eqref{hyp:P1 Lipschitz} will come into play in \Cref{sec:proof of main theorem}.

\subsubsection{Time continuity}
Taking $v_s = \psi_{s} / \eta^{(1)}$ and $\widehat v_s = \psi_{s+h} / \eta^{(1)}$ where $\eta^{(1)} = (1 + |x|^2)^{\frac 1 2}$, we similarly deduce going back to \eqref{eq:dual E known v continuous dependence general} that
\begin{equation}
\label{eq:dual E known time continuity 1}
\begin{aligned}
     \left\| \frac{\psi_{s+h}- \psi_s}{1 + |x|} \right\|_{L^\infty} \le   C_1  \left\| \frac{\psi_{h}- \psi_0}{1 + |x|} \right\|_{L^\infty} + C_2  \sup_{\sigma \in [0,T]} \left\| \frac{\bm E_{T-(\sigma+h)} - \bm E_{T-\sigma}}{1 + |x|} \right\|_{L^\infty}  
\end{aligned} 
\end{equation}
Therefore, $\psi$ inherits the time continuity of $E$.
Then the only remaining difficulty is the time continuity at $0$. 

For $D > 0$ we use Duhamel's formula for the heat equation $u_t - D \Delta u = f$, where we denote the heat kernel $K_D$. Notice that $K_D(t,z) = K_1 (Dt,z) $. For the first term we have that
\begin{align*}
    \frac{\psi_s(x) - \psi_0(x)}{1+|x|}
    =&\, \frac{1}{1+|x|} \Bigg(  \int_\Rd {K_D(s,x-y) } {\psi_0(y)} \diff y - \psi_0(x) \\
   &\qquad\qquad
   +  \int_0^s \int_\Rd K_D(s - \sigma, x-y)  {E_{T-\sigma}(y)} \cdot \nabla \psi_\sigma (y) \diff y \diff \sigma   \Bigg) \\
    =&\, \int_\Rd {K_D(t,x-y) } \frac {\psi_0(y) - \psi_0(x) }{1+|x|} \diff y
    \\ 
    & +  \int_0^s \int_\Rd K_D(s - \sigma, x-y) \frac{1+|y|}{1+|x|} \frac{E_{T-\sigma}(y)}{1+|y|} \cdot \nabla \psi_\sigma (y) \diff y \diff \sigma \,.
\end{align*}
We estimate as follows
\begin{align*}
    \left| \int_\Rd {K_D(s,x-y) } \frac {\psi_0(y) - \psi_0(x) }{1+|x|} \diff y \right|& \le \| \nabla \psi_0 \|_{L^\infty(\Rd)} \int_\Rd K_D(s,x-y)|x-y| \diff y \\
    &=  \| \nabla \psi_0 \|_{L^\infty(\Rd)}  \int_\Rd |z| \frac{e^{-\frac{|z|^2}{4Ds}}}{(4\pi Ds)^{\frac d 2}}  \diff z \\ 
    &= C \| \nabla \psi_0 \|_{L^\infty(\Rd)} (4Ds)^{\frac 1 2} \int_\Rd |z| \frac{e^{-|z|^2}}{\pi^{\frac d 2}}  \diff z.
\end{align*}
Now since $E(y)/(1+|y|)$ and $\nabla \psi$ are bounded, it leaves to integrate
\begin{align*}
    \int_0^s \int_\Rd K_D(\sigma, x-y) \frac{1+|y|}{1+|x|} \diff y \diff \sigma 
    &= \int_0^s \int_\Rd \frac{1}{(4\pi D \sigma)^{\frac d 2}} e^{-\frac{|z|^2}{4 D \sigma }} \frac{1+|x|+|z|}{1+|x|} \diff z \diff \sigma\\
    &= \int_0^s \int_\Rd \frac{1}{(4\pi D \sigma)^{\frac d 2}} e^{-\frac{|z|^2}{4 D \sigma }} (1 + |z|) \diff z \diff \sigma\\
    &= \int_0^s \int_\Rd \frac{1}{(4\pi D \sigma)^{\frac d 2}} e^{-\frac{|z|^2}{4 D \sigma }}  \diff z \diff \sigma + \int_0^s \int_\Rd \frac{1}{(4\pi D \sigma)^{\frac d 2}} e^{-\frac{|z|^2}{4 \sigma }}  |z| \diff z \diff \sigma\\
    &= \int_0^s \int_\Rd \frac{1}{\pi^{\frac d 2}} e^{-{|w|^2}}  \diff w \diff \sigma +  \int_0^s \int_\Rd (2D\sigma)^{\frac 1 2} \frac{1}{\pi^{\frac d 2}} e^{-|w|^2} |w| \diff w \diff \sigma\\
    &\le s + C D^{\frac 1 2} s^{\frac 3 2}.
\end{align*}
Eventually, we recover
{%
using \eqref{eq:dual grad estimate} that
}
\begin{equation}
\label{eq:dual E known time continuity 2}
    \left\| \frac{\psi_s(x) - \psi_0(x)}{1+|x|} \right\|_{L^\infty} \le C(d) \left ( (Ds)^{\frac 1 2}  + (s + D^{\frac 1 2} s^{\frac 3 2}) \sup_{[0,s]} \left \| \frac{ \bm E_{T-\sigma} }{1+|x|} \right\|_{L^\infty}  \right) \| \nabla \psi_0 \|_{L^\infty} 
\end{equation}
This result can also be deduced for $D = 0$ without involving any convolution. Also, it is recovered as a limit $D \searrow 0$.
Joining \eqref{eq:dual E known time continuity 1} and \eqref{eq:dual E known time continuity 2} we recover \eqref{eq:dual time estimate}.

\begin{remark}[Time continuity at $s=0$ if $\Delta \psi_0$ is bounded or $D = 0$]
To estimate the time derivative at time $0$, we consider the candidate sub and super-solutions
\begin{equation*}
    \underline \psi_s = \psi_{0} - C_0 s, \qquad \overline \psi_s = \psi_{0} + C_0 s.
\end{equation*}
We have that
\begin{align*}
    &\partial_s \underline \psi_s - \bm E_{T-s} \cdot  \nabla \underline \psi_{s} - \diffc \Delta \underline \psi_s = - C_0 - \bm E_{T-s} \cdot  \nabla \psi_{0} - \diffc \Delta \psi_0, \\
    &\partial_s \overline \psi_s - \bm E_{T-s} \cdot  \nabla \overline \psi_{s} - \diffc \Delta \overline \psi_s = C_0 - \bm E_{T-s} \cdot  \nabla \psi_{0} - \diffc \Delta \psi_0\,.
\end{align*}
So we need
\begin{equation*}
    C_0 \ge \sup_{[0,T] \times \Rd} |\bm E_{T-s} \cdot  \nabla \psi_{0} + \diffc \Delta \psi_0|.
\end{equation*}
Then
\begin{equation*}
    \left\| \frac{\psi_s - \psi_0}{s}  \right\|_{L^\infty (\Rd)} \le \| \bm E_{T-s} \cdot  \nabla \psi_{0} + \diffc \Delta \psi_0 \|_{L^\infty([0,T] \times \Rd)} .
\end{equation*}
As $s \to 0$, we get an estimate of the time derivative at $s = 0$. 
A similar computation can be done for $\psi / (1 + |x|)$.
This kind of result is useful for the case $D = 0$, since even as $D_k \searrow 0$, one can take approximate initial data $\psi_0^{(k)}$ so that the constant is uniformly bounded.
\end{remark}

\subsection{Proof of \Cref{prop:dual existence Lipschitz}}

\label{sec:proof of prop:dual prop:dual existence Lipschitz}

{%
First we prove uniqueness.
}%
Take $v^{(2)} = \psi / \eta^{(2)}$ where $\eta^{(k)}(x) = (1 + |x|^2)^{k/2}$. We observe that if $\frac{\psi_s}{1+|x|} \in L^\infty$ then $ v_0^{(2)} \in L^\infty(\Rd)$, with decay $1/|x|$ at infinity. 
By dividing viscosity test functions by $\eta^{(2)}$, we observe that $v^{(2)}$ is a viscosity solution of \eqref{eq:dual problem E known v}. 
Notice since $v^{(2)}$ remains bounded, $\sup_{(0,T) \times \mathbb R^d \setminus B_R } \psi / (1 + |x|)$ is bounded by $R^{-\frac 1 2}$.
The uniqueness of $\psi$ follows from the uniqueness of $v^{(2)}$ (see, e.g.,  \cite[Theorem 3.1]{Nunziante1992}). 

{%
Let us now construct the solution and prove its properties.
}%
If $\diffc > 0$, $\psi_0 \in W_c^{2,\infty} (\Rd)$, and $\bm E$ is very regular, then existence is simple by standard arguments.
{%
For the problem in $\Rd$,
regularity 
and decays as $|x| \to \infty$ for compactly supported initial data
}%
follow as in \cite{Evans1998} and hence all the estimates above are justified.
{%
If $\psi_0 \ge 0$ we can proceed like in \Cref{lem:GLP L infty estimate} to prove $\psi \ge 0$.
}

Now we consider the general setting, and we argue by approximation. Consider an approximating sequence $0 < D^{(k)} \to  D$, satisfying
\begin{align*}
    W_c^{2,\infty} (\Rd) \ni \frac{\psi_0^{(k)}}{(1+|x|^2)^{\frac 1 2}} &\longrightarrow   \frac{\psi_0}{(1+|x|^2)^{\frac 1 2}} \quad  \text{ in } L^\infty (\Rd), \\[3mm]
    W^{2,\infty} ([0,T] \times \Rd )^d \ni \bm E^{(k)} &\longrightarrow   \bm E \qquad \qquad  \text{ in } C ((0,T) \times \Rd)^d.
\end{align*}

We have the uniform continuity estimates \eqref{eq:dual grad estimate} and \eqref{eq:dual time estimate}. Due to the Ascoli-Arzelá theorem and a diagonal argument, $\psi^{(k)} \to \psi$ uniformly over compacts of $[0,T] \times \Rd$. 
Applying \Cref{lem:viscosity solutions stability}, the limit is a viscosity solution with the limit coefficients. 
By the uniqueness above this is the solution we are studying.
All estimates are stable by convergence uniformly over compact sets.
\qed

\section{Dual-viscosity solutions of problem \texorpdfstring{\eqref{eq:main E known}}{(PE)}}
\label{sec:main E known}

Now we focus on showing well-posedness and estimates for \eqref{eq:main E known}. We construct the solutions as the duals of those in \Cref{sec:solutions of PE*}.

\begin{proposition}
\label{prop:main E known continuous dependence}
    For every $\diffc, T \ge 0$ and $\bm E \in C([0,T] ; \Lip_0 (\Rd, \Rd)$, and $\mu_0 \in \mathcal P_1 (\Rd)$ there exists exactly one dual viscosity solution $\mu \in C([0,T]; \mathcal P_1 (\Rd))$ of \eqref{eq:main E known}.
    It is also a distributional solution.
    Furthermore, the map 
    \begin{align*}
        S_T :  (\mu_0, \bm E)  \in  \mathcal P_1 (\Rd) \times C([0,T] ; \Lipschitz_0(\Rd, \Rd))  &\longmapsto \mu_T \in \mathcal P_1 (\Rd)
    \end{align*}
    is continuous with the following estimate
    \begin{equation}
        \begin{aligned}
            \label{eq:ST P1 weigth 1 + |x|}
             d_1 (S_T[\mu_0, \bm E], S_T[\widehat \mu_0, \widehat {\bm E}]) 
             &\le C\left( d_1 (\mu_0, \widehat \mu_0) + \int_0^T  \left\| \frac{\bm E_{\sigma} - \widehat {\bm E}_{\sigma}}{1+|x|}  \right\|_{L^\infty} \diff \sigma \right)\\
             \text{where }C &= C\left( d, \diffc,
            \int_0^T \left\|  \nabla {\bm E}_{\sigma}   \right\|_{L^\infty} \diff \sigma, \int_\Rd |x| \diff \mu_0 \right)
        \end{aligned}
    \end{equation}
    depends monotonically on each entry.
    The semigroup property holds in the sense that
    $$
        S_{\widehat t + t} [\mu_0, {\bm E}] = S_{\widehat t} \Bigg[S_{ t} \Big[\mu_0, {\bm E}|_{[0, t]}\Big], {\bm E}|_{[t, t+\widehat t]} (\cdot-t) \Bigg] .
    $$
    Therefore, we have constructed a continuous flow
    \begin{equation*}
            S : (\mu_0, {\bm E})  \in \mathcal P_1 (\Rd) \times C([0,\infty) ; \Lipschitz_0(\Rd, \Rd))  \longmapsto \mu \in C([0,\infty);\mathcal P_1 (\Rd)) .
    \end{equation*}
\end{proposition}

\begin{proof}[Proof of \Cref{prop:main E known continuous dependence}]

We begin by proving uniqueness and continuous dependence. For $D \ge 0$ and any two weak dual viscosity solutions $\mu$ and $\widehat \mu$ corresponding to $(\mu_0, {\bm E})$ and $(\widehat \mu_0, \widehat {\bm E})$ we have, applying \Cref{lem:continuous dependence}, \eqref{eq:dual psi/v bound}, and
\eqref{eq:dual continuous dependence on E - hat E / (1 + |x|)}, that
\begin{equation*}
\begin{aligned}
    d_1 (\mu_T, \widehat \mu_T) 
    &\le 
    	 d_1(\mu_0 , \widehat \mu_0) \sup_{\Lip(\psi_0) \le 1} \Lip(\psi_T^T) 
    	+ \left( 1 + \int_\Rd |x| \diff \widehat \mu_0(x) \right)  
    		 \sup_{ \substack{ \Lip(\psi_0) \le 1 \\ \psi_0(0) = 0  }} 
    		 \left\|  \frac{\psi_T^T - \widehat \psi_T^T}{{1+|x|}} \right \|_{L^\infty} \\
    &\le 
    	C(d) d_1(\mu_0 , \widehat \mu_0) 
    	\exp \left( 
    		\int_0^T \left( \diffc +  \left\| \nabla {{\bm E} _\sigma } \right\|_{L^\infty}  
   		\right) 
   		\diff \sigma \right)\\
    &\quad 
    	+
    	C\left( d, \diffc, \int_0^T \left\|  \nabla {\bm E}_{T-\sigma}   \right\|_{L^\infty} \right)
    	\left( 
    		1 + \int_\Rd |x| \diff \widehat \mu_0(x) 
    	\right)    
    	\int_0^T 
    		 \left\| \frac{{\bm E}_{T-\sigma} - \widehat {\bm E}_{T-\sigma}}{1+|x|}  \right\|_{L^\infty} 
    	\! \! \! \!\diff \sigma .
\end{aligned}
\end{equation*}
The second constant 
obtained from \eqref{eq:dual continuous dependence on E - hat E / (1 + |x|) b} by taking the supremum when $\Lip(\psi_0) \le 1$.
Eventually, we can simplify this expression to \eqref{eq:ST P1 weigth 1 + |x|}. Therefore, we have uniqueness of dual viscosity solutions for any $D \ge 0$.

For $\diffc > 0, \mu_0 \in H^1(\Rd)$ and ${\bm E}_t \in C([0,T]; W^{2,\infty} (\Rd; \Rd))$, the theory of existence and regularity of weak solutions is nowadays well known (see, e.g., \cite{Boccardo2003,Boccardo2021} and the references therein).
The dual problem is decoupled from \eqref{eq:main E known}, and we have already constructed the unique viscosity solutions of the dual problem (see \Cref{prop:dual existence Lipschitz}), that are also weak solutions when ${\bm E}$ is regular. For regular enough datum $\psi_0$, we can use $\psi$ as test function in the weak formulation of \eqref{eq:main E known} to deduce \eqref{eq:duality}. Hence, any weak solution of \eqref{eq:main E known} is a dual viscosity solution.

Let us now show the time continuity with respect to the $\mathcal P_1 (\Rd)$ distance in space. We take $\psi_0$ such that $\Lip(\psi_0) \le 1$, and we estimate
\begin{align*}
\left| \int_\Rd {\psi_0} \diff (\mu_{t+h} - \mu_t) \right| \le \left| \int_\Rd ( \psi^{t+h}_{t+h} - \psi^t_t ) \diff \mu_0 \right| \le \int_\Rd(1 + |x|) \diff \mu_0 \left\| \frac{\psi_{t+h}^{t+h} - \psi_t^t}{1+|x|} \right \| _{L^\infty (\Rd)}.
\end{align*}
This last quantity is controlled by continuous dependence on ${\bm E}$ and time continuity of \eqref{eq:dual E known}. First, letting $\widehat \psi_s = \psi^{t+h}_s$
with $\psi^{t+h}_0 = \psi_0$, which corresponds to $\widehat {\bm E}_s = {\bm E}_{t+h-s}$, we have that
\begin{equation*}
    \left\| \frac{\psi^{t+h}_{t} - \psi_t^t}{1+|x|} \right \| _{L^\infty (\Rd)} 
    \le 
    C 
    \sup_{\sigma \in [0,t]} \| {\bm E}_{t+h - \sigma} - {\bm E}_{t-\sigma}\|_{Lip_0}.
\end{equation*}
The right-hand side of this equation is a modulus of continuity, which we denote $\omega_E$.
Now we use the time continuity of \eqref{eq:dual E known} given by \eqref{eq:dual time estimate} to deduce that
\begin{equation*}
    \left\| \frac{\psi^{t+h}_{t+h} - \psi_t^{t+h}}{1+|x|} \right \| _{L^\infty (\Rd)} \le \omega_D (h),
\end{equation*}
where $C$ and $\omega_D$ are given by the right-hand side of \eqref{eq:dual time estimate}.
The constants are uniform for $\psi_0$ with $\Lip(\psi_0) \le 1$. Eventually, taking the supremum on $\psi_0$ and applying \eqref{eq:duality characterisation of d1} we deduce
\begin{equation*}
    d_1 (\mu_{t+h},\mu_t) \le \omega_D(h) + \omega_E(h).
\end{equation*}

Let us now consider the general case for $\diffc \ge 0$, $\mu_0$ and ${\bm E}_0$. The uniform estimates \eqref{eq:ST P1 weigth 1 + |x|} shows that, if $0 < \diffc^n \to \diffc, H^1(\Rd) \ni \mu_0^n \to \mu_0$ in $\mathcal P_1 (\Rd)$ and $C([0,T]; W^{2,\infty} (\Rd; \Rd)) \ni {\bm E}^n \to {\bm E}$ in $C([0,T]; \Lipschitz_0 (\Rd,\Rd))$, then the sequence $\mu^n$ is Cauchy in the metric space $C([0,T]; \mathcal P_1 (\Rd))$. Since this is a Banach space, the sequence $\mu^n$ converges to a unique limit $\mu$. 
Due to the stability of the dual problem, we can pass to the limit in \eqref{eq:duality} to check that $\mu$ is a dual viscosity solution. We have already shown the uniqueness at the beginning of the proof.
Due to the approximation, the solution constructed is also a distributional solution.

The semigroup property holds in the regular setting, so we can pass to the limit. This completes the proof.
\end{proof}

\section{Existence and uniqueness for \texorpdfstring{\eqref{eq:main}}{(P)}. Proof of \texorpdfstring{\Cref{thm:well-posedness in P1}}{Theorem \ref{thm:well-posedness in P1}}}
\label{sec:proof of main theorem}

We prove the existence by using Banach's fixed-point contraction theorem.
We construct the map
\begin{equation}
    \mathcal T_t : C([0,t] ; \mathcal P_1 (\Rd)^n ) \longrightarrow \mathcal P_1 (\Rd)^n
\end{equation}
by 
\begin{equation*}
    \mathcal T_t[\vec \mu] = \begin{pmatrix}  S_t[   \mu_0^1 ,  {\bm {\mathfrak K}}^1 [\vec \mu] ] \\  \vdots 
    \\ 
    S_t[   \mu_0^n ,  {\bm {\mathfrak K}}^n [\vec \mu]]
    \end{pmatrix}.
\end{equation*}

We must work on the bounded
cubes $Q(R) = B_{\mathcal P_1} (R)^n$ where we recall the definition of ball in Wasserstein space given by \eqref{eq:ball in Wasserstein}.
Hence, we take 
$$R > \max_{i=1, \cdots, n}  d_1(\mu_0^i , \delta_0).$$
To get a very rough estimate of $d_1 (\mathcal T_T[\bm \mu], \delta_0)$ we first indicate that when ${\bm E} = 0$ we recover the heat kernel at time $\diffc t$
\begin{equation*}
    S_t[\delta_0,0] (x) = H_{\diffc t} (x) =  \frac{1}{(4 \pi \diffc t )^{-\frac d 2}} \exp \left (  -\frac{|x|^2}{4\diffc t}  \right) .
\end{equation*}
Hence, using the triangle inequality
\begin{align*}
    d_1 (\delta_0, S_T[\mu_0, {\bm E}] ) &\le d_1(\delta_0, S_T[\delta_0,\bm 0]) 
    + d_1 (S_T[\delta_0, \bm 0], S_T[\mu_0, {\bm E}]).
\end{align*}
For the first term, we simply write
\begin{equation*}
    d_1(\delta_0, S_T[\delta_0,\bm 0]) = d_1 (\delta_0, H_{\diffc t}) \le \omega(\diffc t).
\end{equation*}
To get a uniform constant in \eqref{eq:ST P1 weigth 1 + |x|} define
\begin{equation}
    C_R (T)= \max_{\substack{ t\in [0,T] \\ i=1, \cdots, n  \\   \vec \mu \in C([0,T]; Q (R)) }} \|\nabla {\bm {\mathfrak K}}^i [\vec\mu]_t \|_{L^\infty(\Rd)} < \infty   
\end{equation}
by the assumptions. As a consequence,
for any $\bf \mu$ such that $d_1 (\mu_t^i, \delta_0) \le R$ for all $i = 1, \cdots, n $ and $t \le T_1$, then
$
 d_1 (\delta_0, \mathcal T_t[\bm \mu]^i)  \le  R
$
for all $i = 1, \cdots, n $
and  $t \le T_1$, i.e.,
\begin{equation*}
    \mathcal T   : C\Big([0,T_1] ; Q(R) \Big) \longmapsto  C\Big([0,T_1] ; Q(R) \Big).
\end{equation*}
Notice that, since we are constructing solutions using $S_T$, the constructed solution is also a distributional solution.

Applying again \eqref{eq:ST P1 weigth 1 + |x|} we have that 
for every $i = 1, \cdots, n$
\begin{align*}
    \sup_{t \in [0,T]} d_1 (\mathcal T^i_t[\vec \mu], \mathcal T^i_t[\widehat{ \vec \mu}] ) 
    & =  
    \sup_{t \in [0,T]} 
    	d_1 \Bigg(
    		S_t\Big[\mu_0^i,  {\bm {\mathfrak K}}^i [\vec \mu]\Big], 
    		S_t\Big[\mu_0^i,  {\bm {\mathfrak K}}^i[\widehat{\vec \mu}]\Big] 
    	\Bigg)\\
    &\le 
    C(T_1,R)
    	\int_0^T 
    		\left \| 
    			\frac { {\bm {\mathfrak K}}^i[\vec \mu]_\sigma 
    			- {\bm {\mathfrak K}}^i[\widehat{\vec \mu}]_\sigma }{1 +|x| } 
    		\right \|_{L^\infty}  
    	\diff \sigma  \\
    &\le 
    C(T_1,R) T L 
    \sup_{\substack { \sigma \in [0,T]  \\ j = 1, \cdots, n   }  }
    	d_1 (\mu^j_\sigma, \widehat \mu^j_\sigma), 
\end{align*}
where $C(T_1,R)$ is recovered again through \eqref{eq:ST P1 weigth 1 + |x|}.
Given that $R$ is fixed, we can find $T_2 < T_1$ so that the map $\mathcal T_{T_2}$ is contracting with the norm
\begin{equation*}
    d^{n,T_2}_1 (\vec \mu, \widehat {\vec \mu} ) = \sup_{\substack { t \in [0,T_2]\\ i=1,\cdots,n}} d_1 (\mu^i_t, \widehat \mu^i_t). 
\end{equation*}
Therefore, we can apply Banach's fixed-point theorem to proof existence and uniqueness for short times. 
Since $\mathfrak K$ is uniformly Lipschitz, we can extend the existence time to infinity applying the classical argument. 
To prove the continuous dependence on $\mu_0$, we apply \eqref{eq:duality P1 and weighted Linfty}, \eqref{eq:dual continuous dependence on E - hat E / (1 + |x|)} and \eqref{hyp:P1 Lipschitz}. This completes the proof of existence and uniqueness.
\qed

\begin{remark}
Following the non-explicit constant in \eqref{eq:dual continuous dependence on E - hat E / (1 + |x|)}, it would be possible to recover quantitative dependence estimates.
\end{remark}

\begin{remark}
[Numerical analysis when $\diffc = 0$: the particle method]
Our notion of solution justifies the convergence of the particle method when $\diffc = 0$. 
The aim of the particle method is to consider an approximation of the initial datum given by finitely many isolated particles 
\begin{equation*}
    \mu_0^{i,N} = \sum_{j=1}^N a^{ijN} \delta_{X_0^{ijN}} .
\end{equation*}
Then, it is not difficult to see that, for $\diffc = 0$ the solution is given by particles
\begin{equation*}
    \mu_t^{iN} = \sum_{j=1}^N a^{ijN} \delta_{X_t^{ijN}} .
\end{equation*}
The evolution of these particles is given by a system of ODEs for the particles
\begin{equation*}
        \partial_t X^{ijN}_t = - \sum_{j} a^{ijN} { {\bm K}}^{i}_t [\vec \mu^N_t ] (X^{ijN}_t)
\end{equation*}
This system is well posed. 
Due to the continuous dependence
\begin{equation*}
    \sup_{t \in [0,T]}  d_1 (\vec \mu_t, \vec \mu_t^N) \le C(T , \vec \mu_0)   d_1 (\vec \mu_0, \vec \mu_0^N)
\end{equation*}
It is easy to see that the finite combinations of Dirac deltas is dense in $1$-Wasserstein distance, and estimates for convergence are well known (see, e.g., \cite{Jabin2014}).

\end{remark}

\section{Gradient flows of convex interaction potentials and \texorpdfstring{$\diffc = 0$}{D=0}}
\label{sec:gradient flows}
The aim of this section is to give an application of our results to the classical aggregation equation
\begin{equation}
\label{eq:aggregation equation}
    \partial_t \mu_t = \diver (\mu_t \nabla W * \mu_t).
\end{equation}
This problem falls into the well-posedness theory developed in \Cref{thm:well-posedness in P1} provided that
$  D^2 W \in L^\infty(\Rd)^{d\times d}$ for the $\mathcal P_1 (\Rd)$ theory. 
In the context of gradient-flow solutions (see, e.g., \cite{Ambrosio2005}) we are able to weaken the hypothesis on $\mathfrak K$.
\begin{theorem}
\label{thm:gradient flows}
    Let $\mu_0 \in \mathcal P_2(\Rd) $ be compactly supported. Assume $W$ is convex and that $W \in C^{1,s}_{loc} (\Rd)$. Then, the gradient flow solution $\mu \in C([0,T]; \mathcal P_2 (\Rd))$ is a dual-viscosity solution.
\end{theorem}

We can approximate $W$ by a sequence of convex functions $W^k$, so that $\nabla W^k \to \nabla W$ uniformly over compacts but satisfies $D^2 W^k \in L^\infty (\Rd) $.
We construct $\mu^k_t$ through \Cref{thm:well-posedness in P1}. Since $\mu_0$ is compactly supported and $W^k$ is convex, if we pick $A_0$ the convex envelope of $\supp \mu_0$, then
\begin{equation*}
    \supp \mu_t^k \subset A_0, \qquad \forall t > 0.
\end{equation*}
This is easy to see, for example, looking at the solution by characteristics. We set ourselves in the hypothesis of \Cref{thm:stability} by proving below that
\begin{enumerate}
    \item $\psi^{k,T}$ for every $\psi_0 \in C_c^\infty (\Rd)$ are uniformly equicontinuous.
    
    \item $\mu^k \to \mu $ in $C([0,T]; \mathcal P_2 ( \Rd))$.
    
    \item $W^k * \mu^k \to W*\mu$ uniformly over compacts. 
\end{enumerate}

\begin{remark}
We point that the result in \Cref{thm:gradient flows} can be extended from \eqref{eq:aggregation equation} to a system of equations coming from the gradient flow of interaction potentials between the different components. To fix ideas, we can treat systems of two species like
\begin{equation*}
   \begin{dcases}
       \partial_t \rho_1 = \diver( \rho_1 \nabla  ( H_1 * \rho_1 +  K * \rho_2 ) \\
        \partial_t \rho_2 = \diver( \rho_2 \nabla  ( H_2 * \rho_2+  K * \rho_1 ) 
   \end{dcases}
\end{equation*}
See, e.g., \cite{DAF18}.
\end{remark}

\paragraph{Regularity in space.} Constructing solutions by characteristics it is well known that
\begin{equation*}
    ({\bm E}_{T-s} (x) - {\bm E}_{T-s} (y)) \cdot (x-y) \ge 0,
\end{equation*}
then characteristics grow apart. Then, the solution of \eqref{eq:dual linear} satisfies 
\begin{equation}
    \label{eq:decay of gradient}
    \| \nabla \psi_s \|_{L^\infty (\Rd)} \le \| \nabla \psi_0 \|_{L^\infty(\Rd)}.
\end{equation}
    We check that this holds for ${\bm E} = \nabla W * \mu$. Since $W$ is convex, then
    \begin{equation*}
        (\nabla W (x) - \nabla W (y)) \cdot  (x-y) = (x-y) \cdot  D^2 W(\xi(x,y)) (x-y) \ge 0.
    \end{equation*}
    Then, the convolution with a non-negative measure is also convex
    \begin{align*}
        \Big(\nabla W * \mu (x) - \nabla W * \mu  (y)\Big) \cdot  \Big(x-y \Big) &= \int_\Rd \Big( \nabla W (x-z) - \nabla W (y-z)  \Big) \cdot \Big( (x-z) - (y-z) \Big) \diff \mu_0 (z) \\
        &\ge 0,
    \end{align*}
    and so \eqref{eq:decay of gradient} holds.

\paragraph{Regularity in time.}
First, we look at the evolution of the support. When ${\bm E}_{T-s}$ is locally bounded, we can construct a super-solutions by characteristics.
We define
\begin{equation*}
    A(s_0,s_1) = \bigcup_{s \in [s_0,s_1]} \supp \psi_s .
\end{equation*}
If we assume that $A(s,0) \subset B(x_0,R_s)$ then 
\begin{equation*}
    \partial_s R_s \le \sup_{B(x_0,R_s)} |{\bm E}_{T-s}|.
\end{equation*}
So we end up with an estimate
\begin{equation}
    \label{eq:dual support estimate}
    A(s,0) \subset \supp \psi_0 + B_{R_s}. 
\end{equation}

Assume that $\psi_0 \in C_c^\infty (\Rd)$ with $\|\nabla \psi_0\|_{L^\infty(\Rd)} \le 1$. Take $s_0 < s$. Consider
\begin{equation*}
    \underline \psi_s = \psi_{s_0} - C_0 (s-s_0), \qquad \overline \psi_s = \psi_{s_0} + C_0 (s-s_0).
\end{equation*}
We have that
\begin{equation*}
    \partial_s \underline \psi_s - {\bm E}_{T-s} \cdot \nabla \underline \psi_{s} = - C_0 - {\bm E}_{T-s} \cdot \nabla \psi_{s_0}, \qquad \partial_s \overline \psi_s - {\bm E}_{T-s} \nabla \overline \psi_{s} = C_0 - {\bm E}_{T-s} \nabla \psi_{s_0}.
\end{equation*}
They are a sub and super-solution in $[s_0,s_1]$ if
\begin{equation*}
    C_0 = \|\nabla \psi_{s_0}\|_{L^\infty}  \sup_{[s_0,s_1] \times A(s_0,s_1) } |{\bm E}|. 
\end{equation*}
Hence, we deduce that
\begin{equation}
\label{eq:dual uniform time continuity}
    \left\| \frac{\psi_{s_1} - \psi_{s_0}}{s_1-s_0} \right\|_{L^\infty (\Rd)} \le  \|\nabla \psi_{s_0}\|_{L^\infty}  \sup_{[s_0,s_1] \times A(s_0,s_1) } |{\bm E}| .
\end{equation}
This implies that a uniform bound on the time continuity based on $T$, local bounds of ${\bm E}$ and the support of $\psi_0$.

\paragraph{Convergence of the convolution.}
Following  \cite[Theorem 11.2.1]{Ambrosio2005} 
the $\Gamma$-convergence of uniformly $\lambda$-convex interaction free energy functionals is sufficient so that
\begin{equation}
\label{eq:grad flow P2 convergence}
    \sup_{[0,T]} d_2 (\mu_t^k , \mu_t)\to 0, \qquad \text{ as } k \to \infty.
\end{equation}
In particular, if $W \in C^s$ is convex and $W^k (x) (1 +|x|)^{-2} \to W(x) (1 + |x|)^{-2}$ uniformly, \eqref{eq:grad flow P2 convergence} holds. 
It follows that the sequence $\mu^k$ is uniformly continuous, i.e., the function
\begin{equation*}
    \omega(h) = \sup_k \sup_{t \in [0,T-h]} d_2 (\mu_{t+h}^k , \mu^k_t)
\end{equation*}
is a modulus of continuity.

On the other hand, for the convergence of $\nabla W^k * \mu^k$ we prove the following result.
\begin{lemma}
\label{lem:grad flows convergence of convolution}
Assume that
\begin{enumerate}
    \item $\nabla W \in C^s_{loc}$
    \item $\nabla W^k \to \nabla W$ uniformly over compacts of $\Rd$.
    \item $\mu_t^k, \mu_t \in \mathcal P (\Rd)$, and, for every $t > 0$, $\mu^k_t \rightharpoonup \mu_t$ weak-$\star$ in $\mathcal M (\Rd)$
    \item There exists $A_0 \subset \Rd$ convex bounded such that for all $t > 0$, $\supp \mu_t^k  \subset A_0$
    \item $\mu^k \in C([0,T]; \mathcal P_2 (\Rd))$ are uniformly continuous.
\end{enumerate}
Then
\begin{equation*}
    \nabla W^k * \mu^k \to \nabla W * \mu \text{ uniformly over compacts of } [0,T] \times \Rd .
\end{equation*}
\end{lemma}

\begin{proof}
We use the intermediate element $\nabla W * \mu^k$. First, for $A \subset \Rd$ compact
\begin{equation*}
    \sup_{x \in A} \left| \int_\Rd (\nabla W^k (x-z) - \nabla W (x-z)  ) \diff \mu_s^k  \right| \le \sup_{A + A_0} |\nabla W^k -\nabla W |.
\end{equation*}
Hence, we have that
\begin{equation}
    \label{eq:approx conv 1}
    \sup_{\substack {t \in [0,T] \\ x \in A } } \left|  \nabla W^k * \mu_t^k (x) - \nabla W * \mu_t^k (x) \right| \le \sup_{A + A_0} | \nabla W^k - \nabla W|.
\end{equation}

Due to weak-$\star$ convergence, if $\nabla W \in C^s_{loc}$ then
\begin{equation*}
    \nabla W * \mu^k_t (x) \to \nabla W * \mu_t (x) \text{ for each }  (t,x) \in [0,T] \times \Rd 
\end{equation*}
Now we prove uniform continuity. First in space. Let $A \subset \Rd$ be compact. Take
\begin{equation*}
    C_K  = \sup_{ \substack { x, y \in A - A_0 \\ x \ne y} } \frac{|\nabla W(x) - \nabla W(y)|}{|x-y|^s}
\end{equation*}
Now, for $x,y \in A$ we can compute
\begin{equation*}
    \left|\nabla W* \mu^k (x) - \nabla W*\mu^k (y) \right| \le \int_{A_0} |\nabla W(x-z) - \nabla W(y-z)| \diff \mu^k (z)\le C_A |x-y|^s.
\end{equation*}
So $\nabla W * \mu^k$ is uniformly continuous in $x$ over compacts of $[0,T] \times \Rd$.

Lastly, let $\pi$ be optimal plan between $\mu^k_t$ and $\mu^k_{\tau}$. Due to assumption 4, the optimal plan  can be selected so that $\supp \pi \subset A_0 \times A_0$.
Let $z \in A$. We have that
\begin{align*}
		\left| \int_\Rd \nabla W(z-y) \diff \mu^k_t(x) - \int_\Rd \nabla W(z-y) \diff  \mu^k_{\tau}(y) \right| 
		&= \left| \iint ( \nabla W(z-x) - \nabla W(z-y) ) \diff \pi(x,y) \right| \\
		&\le C_A \iint |x-y|^s \diff \pi(x,y) \\
		&\le C_A \left( \iint |x-y|^2 \diff \pi(x,y) \right)^{\frac s 2}\\
        &= C_A d_2 (\mu^k_t, \mu^k_{\tau})^s \le C_A \omega (h)^s.
\end{align*}
Hence, we have the uniform estimate of continuity in time
\begin{equation*}
    \sup_{x \in K} |\nabla W * \mu^k_t - \nabla W * \mu^k_{\tau}| \le C_A \omega(|t-\tau|)^s. 
\end{equation*}
And we finally deduce that
\begin{equation*}
    \sup_{\substack{ t,\tau \in [0,T] \\ x, y \in K}} |\nabla W * \mu^k_t (x) - \nabla W * \mu^k_\tau (y)| \le C_A (|x-y| + \omega (|t - \tau|)^s).
\end{equation*}
In particular, by the Ascoli-Arzelá theorem, there is a subsequence converging uniformly in $[0,T]$. Since we have characterised the point-wise limit, every convergent subsequence does so to $\nabla W * \mu$. Hence, the whole sequence converges uniformly over compacts, i.e.,
\begin{equation}
\label{eq:approx conv 2}
        \sup_{\substack{ t \in [0,T] \\ x \in A}} |\nabla W * \mu^k_t (x) - \nabla W * \mu_t (x)| \to 0 , \qquad \text{as } k \to \infty.
\end{equation}
Using the triangular inequality, \eqref{eq:approx conv 1}, and \eqref{eq:approx conv 2} the result is proven.
\end{proof}

\begin{proof}[Proof of \Cref{thm:gradient flows}]
When $W^k$ is $C^2$, we can construct a unique classical solution as the push-forward of regular characteristics. This solution is well-known to be the gradient flow solution. 
By construction, it coincides with the unique dual viscosity solution that exists by \Cref{thm:well-posedness in P1}.

First, we showed in \eqref{eq:grad flow P2 convergence} that $\mu^k$ converges to the gradient flow solution in $C([0,T]; \mathcal P_2 (\Rd))$, and let us denote it by $\overline \mu$.
Now we apply \Cref{thm:stability} 
where the hypothesis have been check in \eqref{eq:dual uniform time continuity} (using \eqref{eq:dual support estimate}), and
\Cref{lem:grad flows convergence of convolution} to show that $\overline \mu$ is a dual viscosity solution. This completes the proof.
\end{proof}

\section{A \texorpdfstring{$\dot H^{-1}$}{dot H-1} and \texorpdfstring{$2$}{2}-Wasserstein theory when \texorpdfstring{$\diffc > 0$}{D>0}}
\label{sec:well-posedness in H-1}
\subsection{Notion of solution and well-posedness theorem}
Many of cases of  \eqref{eq:main} studied in the literature  are $2$-Wasserstein gradient flow. Our situation is more general. Unfortunately, the 2-Wasserstein distance, $d_2$, does not have a duality characterisation similar to $d_1$. 
It is known (see, e.g., \cite{Peyre2018}) that it can be one-side compared with the $\dot H^{-1}(\Rd)$
\begin{equation}
    \label{eq:H-1 embedding W2}
    d_2 (\mu, \widehat \mu) \le 2 [ \mu - \hat \mu ]_{\dot H^{-1} (\Rd)}, 
\end{equation}
where, for $\mu \in D' (\Rd)$ we define the norm 
\begin{equation*}
    \|\mu\|_{\dot H^{-1}(\Rd)} = \sup_{ \substack{ f \in C_c^\infty (\Rd) \\  \| \nabla f \|_{L^2} \le 1} }   |\mu (f)|.
\end{equation*}
The converse inequality to \eqref{eq:H-1 embedding W2} only holds for absolutely continuous measures, and the constant depends strongly on the uniform continuity.

We consider the Sobolev semi-norm 
$
    [f]_{H^1} = \|\nabla f\|_{L^2}.
$
The space $(C_c^\infty (\Rd) , [ \cdot ]_{H^1})$ is a normed space. Notice that, if $[f]_{H^1} = 0$ then $f$ is constant. But since it has compact supported, the value of the constant is $0$. This allows to define the dual space
\begin{equation*}
    \dot H^{-1} (\Rd) = (C_c^\infty(\Rd), [ \cdot ]_{H^1(\Rd)})'.
\end{equation*}
Since it is the dual of a normed space, $\dot H^{-1} (\Rd)$ is a Banach space. 

\begin{remark}
    The space $\dot H^1(\Rd)$ is defined as the completion of $(C_c(\Rd), [ \cdot ]_{H^1(\Rd)})$, which is easy to see is not complete itself. This completion can be complicated (see, e.g., \cite{Brasco2021}). 
    Hence, with our construction $\dot H^{-1} (\Rd)$ is not  the dual of $\dot H^{1} (\Rd)$. 
\end{remark}

Similarly to above, we define
\begin{definition}
	We say that $(\vec \mu, \{\vec \Psi^T\}_{T\ge 0})$ is an $\dot H^{-1}$ entropy pair
	if:
	\begin{enumerate}
		\item For every $T \ge 0$, 
		$$
		    \vec \Psi^T: X = \{ \psi_0 \in C_c(\Rd) : \nabla \psi_0 \in L^2 (\Rd) \} \longrightarrow C([0,T] ; X )^n
		$$ 
		is a linear map with the following property: for every $\psi_0$ and $i = 1, \cdots, n$ we have $\Psi^{T,i}[\psi_0]$ is a viscosity solution of \eqref{eq:dual}.
		\item For each $i$ and $T \ge 0$, $\mu_T^i \in \dot H^{-1} (\Rd)$ and satisfies the duality condition \eqref{eq:duality}
	\end{enumerate}
\end{definition}
\begin{remark}
    \label{rem:dot H -1 compact support}
    Notice that in this section we require that $\psi_0$ is compactly supported, and thus the unique viscosity solutions will satisfy $\psi(x) \to 0$ as $|x| \to \infty$.
\end{remark}

The main result of this section is
\begin{theorem}
    \label{thm:well-posedness in H-1}
    Let $\diffc > 0$, $\mu_0 \in \dot H^{-1} (\Rd)$ and assume that 
    \begin{equation}
    \label{eq:K defined from dot H1}
        \bm {\mathfrak K}^i : C([0,T],  \dot H^{-1} (\Rd)) \to C([0,T],  W^{1,\infty} (\Rd))
    \end{equation}
    is Lipschitz with data in $ \dot H^{-1} (\Rd)$ in the sense that, for any $\vec \mu$ and $\widehat{\vec \mu}$ in $C([0,T]; \dot H^{-1} (\Rd))$
    \begin{equation}
    \label{eq:K Lipschitz from dot H1}
         \sup_{\substack{t \in [0,T] \\ i = 1, \cdots, n}}  \left\|   {\bm {\mathfrak K}^i[\vec \mu]_t - \bm {\mathfrak K}^i[ \widehat {\vec \mu}]}_t  \right\|_{W^{1,\infty}(\Rd)} \le L \sup_{\substack{t \in [0,T] \\ i=1,\cdots,n }} \| \mu^i_t - \widehat \mu ^i_t \|_{\dot H^{-1}(\Rd)}.
    \end{equation}
    Then there exists exactly one dual viscosity solution $\vec \mu  \in C([0,T], \dot H^{-1} (\Rd)^n)$.
    
    Furthermore, if $\bm {\mathfrak K}^i$ is autonomous (i.e., $\bm {\mathfrak K}^i[\vec \mu]_t = {\bm K}^i[\vec \mu_t]$), then the map $S_T : \vec \mu_0 \in \dot H^{-1} (\Rd)^n \mapsto \vec \mu_T \in \dot H^{-1} (\Rd) ^n$ is a continuous semigroup.
\end{theorem}
We point that we can only get the $\dot H^{-1}(\Rd)$ theory when diffusion is present (i.e., $\diffc > 0$). %
{%
When $D = 0$ we cannot get an estimate of $\|\nabla \psi_s\|_{L^2}$.
}
\begin{proof}[Proof of \Cref{thm:well-posedness in H-1}] 
    The proof follows the blueprint of the proof of \Cref{thm:well-posedness in P1} using a fixed-point argument. We need an adapted version of \eqref{eq:duality P1 and weighted Linfty} given by
    \begin{equation} 
    \label{eq:H-1 continuous dependence by duality}
    \begin{aligned}
        \|\mu_T - \widehat \mu_T\|_{\dot H^{-1}(\Rd)} 
        \le &\, \|\mu_0 - \widehat \mu_0\|_{\dot H^{-1}(\Rd)} \sup_
        {\substack{ \psi_0 \in C_c^\infty (\Rd) \\ \| \nabla \psi_0 \|_{L^2 (\Rd)}  \le 1 }} \| \nabla \psi_T^T \|_{L^2(\Rd)}  \\
        & +  \|\widehat \mu_0\|_{\dot H^{-1}(\Rd)} \sup_{\substack{ \psi_0 \in C_c^\infty (\Rd) \\ \| \nabla \psi_0 \|_{L^2 (\Rd) }  \le 1 }}  \| \nabla( \psi_T^T - \widehat \psi_T^T ) \|_{L^2(\Rd)}  .
    \end{aligned}
    \end{equation}
    This is shown by a suitable modification of the argument in \eqref{eq:duality P1 and weighted Linfty deduction}. Lastly, we use the bounds and continuous dependence proved below in \Cref{prop:dual L2 estimates,prop:E known in H-1}.
\end{proof}

Lastly, due to \eqref{eq:H-1 embedding W2} we have the following partial result in $2$-Wasserstein space.
\begin{corollary}
  Let $\diffc > 0$, $\vec \mu_0 \in \dot H^{-1} (\Rd)^n \cap \mathcal P(\Rd)^n $, then the unique solution constructed in \Cref{thm:well-posedness in H-1} is $C([0,T]; \mathcal P_2 (\Rd)^n)$.
\end{corollary}
\begin{proof}
    Due to the previous theorem, we only need to show that probability distributions stay probability distributions. 
    The non-negativity is trivial, since by \Cref{prop:dual existence Lipschitz} the test functions preserve the non-negativity. 
    Hence, we get
    \[ 
        \int_\Rd \psi_0  \diff \mu_T = \int_\Rd \psi_T \diff \mu_0 \ge 0, \qquad \text{ for all } 0 \le \psi_0 \in C_c^\infty(\Rd).
    \] 
    Thus $\mu_T \ge 0$ for all $T \ge 0$.
    
    We now prove the conservation of total mass.
    Consider a sequence of $0 \le \psi_{0}^{(k)} \in X$ point-wise increasing $k$ and converging uniformly over compacts to $1$. 
    Then, by the comparison principle, $0 \le \psi^{(k)} \le 1$ are point-wise increasing in $k$.
    By Dini's theorem, functions $\psi^{(k)}$ converges uniformly over compacts.
    Let $\psi^{(\infty)}$ be its limit.
    By stability of viscosity solutions $\psi^{(\infty)}$ solves the same equation.
    By the uniform convergence $\psi^{(\infty)}_0 = \lim_n \psi_0^{(k)} = 1$. 
    Then $\psi^{(\infty)} = 1$. 
    By definition
    \[ 
        \int_{\mathbb R^d} \psi^{(k)}_0  d \mu_T  = \int_{\mathbb R^d} \psi^{(k)}_T d \mu_0 .
    \]
    We can apply the monotone convergence theorem on both sides to deduce
    \[ 
        \int_{\mathbb R^d} d \mu_T  = \int_{\mathbb R^d} d \mu_0  = 1.\qedhere
    \] 
\end{proof}

\subsection{Study of \texorpdfstring{\eqref{eq:dual E known}}{PE*}}

To deal with the $\dot H^{-1}(\Rd)$ estimates, we must get $L^2(\Rd)$ of $\nabla \psi$. 

\begin{proposition}
\label{prop:dual L2 estimates}
Under the hypothesis of \Cref{prop:dual existence Lipschitz} together with $\psi_0 \in C_c (\Rd)$, then the unique viscosity solution of \eqref{eq:dual linear} constructed in \Cref{prop:dual existence Lipschitz} also satisfies 
\begin{equation}
\label{eq:dual grad psi in L2}
    \|\nabla \psi_s\|_{L^2(\Rd)} \le \|\nabla \psi_0\|_{L^2(\Rd)} \exp \left(  C(d) \int_0^s \| \nabla {\bm E}_{T-\sigma} \|_{L^\infty(\Rd)} \diff \sigma  \right) .
\end{equation}
Lastly, given $\widehat \psi_0 = \psi_0$ and $\widehat {\bm E}$ in the same hypotheses above, there exists a corresponding solution of \eqref{eq:dual linear}, denoted by $\widehat \psi$,
and we have the continuous dependence estimate
\begin{equation}
\label{eq:dual grad psi L2 cont dep}
    \| \nabla (\psi_s - \widehat \psi_s) \|_{L^2(\Rd)}^2 \le C_1 e^{C_2 s} \left( 1 + \frac 1 \diffc \right) \int_0^s \| {\bm E}_{T-\sigma} - \widehat {\bm E}_{T-\sigma}  \|_{W^{1,\infty}(\Rd) } \diff \sigma.
\end{equation}
\end{proposition}

\begin{remark}
	Notice that for $\nabla \psi_s \in L^2(\Rd)$ we do not use that ${\bm E} \in L^\infty(\Rd)$, or the ellipticity constant.
\end{remark}

\begin{proof} 
We begin with the $L^2(\Rd)$ estimates for smooth $\psi_0 \in C_c^\infty (\Rd)$ and ${\bm E} \in C_c^\infty ([0,T] \times \Rd)$, where the solutions are classical and differentiable.
For our duality characterisation in $\dot H^{-1}(\Rd)$ we do not need $L^2(\Rd)$ estimates on $\psi$, only on $\nabla \psi$. 
Notice that \eqref{eq:dual grad psi in L2} is simply \eqref{eq:psi derivative Lp estimate} when $p = 2$.

For the continuous dependence, we note that
\begin{align*}
    \partial_s (U^i_s - \widehat U^i_s) = &\, \nabla (U^i_s - \widehat U^i_s ) {\bm E}_{T-s} + \nabla (\psi_s - \widehat \psi_s) \frac{\partial {\bm E}}{\partial x_i} \\
    & +\nabla \widehat U^i_s ({\bm E}_{T-s} - \widehat {\bm E}_{T-s}) + \nabla \widehat \psi_s \frac{\partial}{\partial x_i} ({\bm E}_{T-s} - \widehat {\bm E}_{T-s}) \\
    & + \diffc \Delta (U^i_s - \widehat U^i_s).
\end{align*}
Multiplying by $U^i_s - \widehat U^i_s$ and integrating, we get
\begin{equation*}
    \partial_s \|U^i_s - \widehat U^i_s\|_{L^2 (\Rd)}^2 = I_1 + \cdots + I_5\,.
\end{equation*}
We estimate term by term the integrals $I_i$, $i=1,\dots,5$, as
\begin{align*}
    \left| I_1 \right|  &= \frac 12  \left| \int  (U^i_s - \widehat U^i_s )^2 \diver {\bm E}_{T-s}  \right| \le\frac 12  \| U^i_s - \widehat U^i_s \|_{L^2(\Rd)}^2 \left\| \diver {\bm E}_{T-s} \right\|_{L^\infty (\Rd)} ,\\
    \left| I_2\right|&\le \sum_j \left| \int  (U^j_s - \widehat U^j_s ) \frac{\partial {\bm E}^j}{\partial x_i} (U^i_s - \widehat U^i_s )\right| \le \| U^i_s - \widehat U^i_s \|_{L^2} \sum_{j} \| U^j_s - \widehat U^j_s \|_{L^2(\Rd)} \left\| \frac{\partial {\bm E}^j}{\partial x_i} \right\|_{L^\infty} ,\\
    \left| I_4 \right| &\le  \sum_j \left| \int   \widehat U^j_s \frac{\partial}{\partial x_i} ({\bm E}_{T-s}^j - \widehat {\bm E}_{T-s}^j) (U^i_s - \widehat U^i_s) \right|\\
    &\le \| U^i_s - \widehat U^i_s \|_{L^2(\Rd)} \sum_j   \left\| \frac{\partial}{\partial x_i} ({\bm E}_{T-s}^j - \widehat {\bm E}_{T-s}^j)  \right\|_{L^\infty} \| \widehat U^j_s \|_{L^2(\Rd)} \\
    &\le  \frac 1 2 \| U^i_s - \widehat U^i_s \|_{L^2(\Rd)}^2 + \frac 1 2 \left(  \sum_j   \left\| \frac{\partial}{\partial x_i} ({\bm E}_{T-s}^j - \widehat {\bm E}_{T-s}^j)  \right\|_{L^\infty(\Rd)} \| \widehat U^j_s \|_{L^2(\Rd)} \right)^2 ,
    \intertext{and}
     I_5 &= - D\int| \nabla (U^i_s - \widehat U^i_s)|^2 .
\end{align*}
There is only one problematic term that requires the use of the ellipticity condition
\begin{equation}
\label{eq:H-1 need of ellipticity}
\begin{aligned} 
    \left| I_3 \right| \le &\, \left| \int U^i_s (U^i_s - \widehat U^i_s) \diver ({\bm E}_{T-s} - \widehat {\bm E}_{T-s})  \right|  
    + \left| \int U^i_s   ({\bm E}_{T-s} - \widehat {\bm E}_{T-s}) \nabla (U^i_s - \widehat U^i_s)  \right| \\
    \le &\,\frac 1 2 \| U^i_s - \widehat U^i_s \|_{L^2(\Rd)}^2 + \frac 1 2\| U^i_s \|_{L^2(\Rd)}^2   \| \diver ({\bm E}_{T-s} - \widehat {\bm E}_{T-s}) \|_{L^\infty(\Rd)}^2 \\
    &  + \frac{1}{4\diffc} \|  U^i_s \|_{L^2(\Rd)}^2 \|  {\bm E}_{T-s} - \widehat {\bm E}_{T-s} \|_{L^\infty(\Rd)}^2 + \diffc \| \nabla (U^i_s - \widehat U^i_s) \|_{L^2(\Rd)}^2.
\end{aligned}
\end{equation}
Notice that $I_5$ cancels out the last term in $|I_3|$.
Arguing as above we recover
\begin{align*}
    \partial_s \| \nabla (\psi_s - \widehat \psi_s) \|_{L^2(\Rd)}^2 \le &\, C(d) \Bigg( (1 + \|\nabla {\bm E}_{T-s} \|_{L^\infty(\Rd)} )  \|\nabla (\psi_s^T - \widehat \psi_s^T) )\|_{L^2(\Rd)}^2 \\
    & + \left( 1 + \frac 1 \diffc \right)  \Big ( \|\nabla \psi_s\|_{L^2(\Rd)}^2  +  \|\nabla \widehat \psi_s\|_{L^2(\Rd)}^2 \Big ) \|{\bm E}_{T-s} - \widehat {\bm E}_{T-s}\|_{W^{1,\infty}(\Rd)}^2\Bigg).
\end{align*}
Eventually, since $\psi_0^T - \widehat \psi_0^T = 0$ we recover \eqref{eq:dual grad psi L2 cont dep} where the constants depend on $d$, $\| \nabla \psi_0 \|_{L^2}$, $\| \nabla \widehat \psi_0 \|_{L^2}$ and $\int_0^T\|{\bm E}_{T-\sigma}\|_{W^{1,\infty}} \diff \sigma$. This completes the $L^2$ estimates.

When $\psi_0$ is a general initial datum in $C_c (\Rd)$ we proceed by an approximation argument as in \Cref{prop:dual existence Lipschitz}. As for the $L^\infty(\Rd)$ estimates, we can first assume that ${\bm E} \in C^\infty_c ([0,T] \times \Rd)$ and $\psi_0 \in C^\infty_c (\Rd)$ are smooth, recover the $L^2$ estimates for the gradient, and then pass to the limit. 
\end{proof} 

\begin{remark}
Notice that in \eqref{eq:H-1 need of ellipticity} we use strongly the fact that $D > 0$, and the estimates are not uniform as $D \searrow 0$.
\end{remark}

\subsection{Study of \texorpdfstring{\eqref{eq:main E known}}{PE}}

Similarly, applying \eqref{eq:H-1 continuous dependence by duality}, \eqref{eq:dual grad psi in L2} and \eqref{eq:dual grad psi L2 cont dep}, we deduce that
 \begin{proposition}
    \label{prop:E known in H-1}
    For every $\diffc, T > 0$, ${\bm E} \in C([0,T]; W^{1,\infty} (\Rd; \Rd))$ and $\mu_0 \in \dot H^{-1} (\Rd)$ there exists exactly one $\dot H^{-1} (\Rd)$ dual viscosity solution $\mu \in C([0,T]; \dot H^{-1} (\Rd))$ of \eqref{eq:main E known}.
    Furthermore, the map 
    \begin{align*}
        S_T :   \dot H^{-1} (\Rd) \times C([0,T] ; W^{1,\infty}(\Rd, \Rd))  &\longmapsto \mu_T \in  \dot H^{-1} (\Rd) 
    \end{align*}
    is continuous with the following estimate
   \begin{equation}
   \label{eq:ST H-1}
    \Big\|S_T [{\bm E}, \mu_0] - S_T[ \widehat {\bm E}, \widehat \mu_0]\Big\|_{\dot H^{-1}} \le C\left(d, \tfrac 1 \diffc, T, \int_0^T \| \nabla {\bm E}_{T-\sigma} \|_{L^\infty} \diff \sigma \right) \int_0^T \| {\bm E} _\sigma - \widehat {\bm E} _\sigma \| _{W^{1,\infty}} \diff \sigma, 
\end{equation}
    where $C$ depends monotonically on each entry.
    Lastly, the semigroup property holds, i.e., 
    $$
        S_{\widehat t + t} [\mu_0, {\bm E}] = S_{\widehat t} \Bigg[S_{ t} \Big[\mu_0, {\bm E}|_{[0, t]}\Big], {\bm E}|_{[t, t+\widehat t]} \Bigg].
    $$
\end{proposition}
 
\begin{remark}[Solutions by characteristics when $\diffc = 0$]
For $\diffc = 0$ our notion of solution is $\mu_t = X_t \# \mu_0$ where $X_t$ is the unique solution of the flow equation
\begin{equation}
    \begin{dcases}
    \frac{\partial X_t}{\partial t} = -{\bm E}_t (X_t), \\
    X_0 (x) = x.
    \end{dcases}
\end{equation}
A unique solution of this pointwise-decoupled problem exists via the Picard-Lindelöf theorem. 
Since we have forwards and backwards uniqueness, for each $t > 0$ we know that $X_t$ is a bijection.
It is easy to show that the unique viscosity solution is given by
\begin{equation*}
    \psi_s^T = \psi_0 \circ X_T^{-1} \circ X_{T-s}
\end{equation*}
And the duality condition \eqref{eq:duality} is trivially satisfied. The problem is that we do not have suitable $H^{1}$ in this theory.
\end{remark}

\section*{Acknowledgments}
The research of JAC and DGC was supported by the Advanced Grant Nonlocal\--CPD (Non\-local PDE\-s for Complex Particle Dynamics: Phase Transitions, Patterns and Synchronization) of the European Research Council Executive Agency (ERC) under the European Union’s Horizon 2020 research and innovation programme (grant agreement No. 883363).
JAC was partially supported by EPSRC grant numbers EP/T022132/1 and EP/V051121/1.
DGC was partially supported by grant %
{ 
RYC2022-037317-I, PID2021-127105NB-I00,%
}
PGC2018-098440-B-I00 from the Ministerio de Ciencia, Innovación y Universidades of the Spanish Government. \\

\noindent \textbf{Data Availability statement.} The manuscript has no associated data.\\
\noindent \textbf{Conflict of interest.} On behalf of all authors, the corresponding author states that there is no conflict of interest.

\printbibliography

\end{document}